\documentclass{article}

\usepackage[utf8]{inputenc}        
\usepackage[T1]{fontenc}             
\usepackage[british]{babel}                 
\usepackage{scrpage2}                            
\usepackage{enumerate}             

\usepackage{bbm}
\usepackage[backend=bibtex, maxnames=10, firstinits=true]{biblatex} 
\usepackage{csquotes}
\usepackage{amssymb}
\usepackage{amsmath}
\usepackage{amsthm}

\DeclareFieldFormat[article]{title}{#1}
\DeclareFieldFormat[inproceedings]{title}{#1}
\DeclareFieldFormat[inproceedings]{booktitle}{\addthinspace In:\addthinspace\emph{#1}}
\DeclareFieldFormat[inbook]{booktitle}{\addthinspace In:\addthinspace\emph{#1}}
\DeclareFieldFormat[inbook]{title}{#1}
\DeclareFieldFormat{pages}{\addcolon\addthinspace#1}
\DeclareFieldFormat{issue}{\mkbibparens{#1}}
\DeclareFieldFormat[article]{number}{\mkbibparens{#1}}
\DeclareFieldFormat{journaltitle}{\emph{#1}\addthinspace}
\DeclareFieldFormat{year}{\addcomma\addthinspace#1}
\DeclareFieldFormat{eprint}{\addcomma\addthinspace {\tt arXiv:#1}}
\DeclareListFormat{publisher}{#1\ifthenelse{\value{listcount}<\value{liststop}}
{\addsemicolon\addthinspace}
{\addcomma\addthinspace}}

\DeclareBibliographyDriver{article}{%
\printnames{author}%
\newunit
\printfield{title}%
\iffieldundef{eprint}{
\newunit
\printfield{journaltitle}%
\printfield{volume}%
\printfield{issue}%
\printfield{number}%
\printfield{pages}%
\printfield{year}%
}{ \newunit\printfield{year}
\printfield{eprint}} 
\finentry}

\DeclareBibliographyDriver{inproceedings}{%
\printnames{author}%
\newunit
\printfield{title}%
\newunit
\printfield{booktitle}%
\printfield{volume}%
\printfield{issue}%
\printfield{number}%
\printfield{pages}%
\printfield{year}%
\finentry}

\DeclareBibliographyDriver{inbook}{%
\printnames{author}%
\newunit
\printfield{title}%
\newunit
\printfield{booktitle}
\newunit
\printlist{publisher}%
\printfield{year}%
\finentry}

\DeclareBibliographyDriver{book}{%
\printnames{author}%
\newunit
\printfield{title}%
\newunit
\printlist{publisher}%
\printfield{year}%
\finentry}

\newtheorem{counterSections}{}[section]

\newtheorem{theorem}[counterSections]{Theorem}
\newtheorem{lemma}[counterSections]{Lemma}

\newtheorem{claim}[counterSections]{Claim}

\newtheorem{remark}{Remark}

\newcommand{\N}{\mathbb{N}}
\newcommand{\R}{\mathbb{R}}
\newcommand{\E}{\mathbb{E}}
\newcommand{\Prob}{\mathbb{P}}
\newcommand{\V}{\mathbb{V}}
\newcommand{\Ind}{\mathbbm{1}}

\newcommand{\cond}{\; \middle\vert \;}
\newcommand{\event}[1]{\mathcal{A}_{#1}}

\title{The phase transition in multi-type binomial random graphs}
\author{Mihyun Kang\footnotemark[1]\ \footnotemark[2]  \and Christoph Koch\footnotemark[1]\ \footnotemark[3] \and Ang\'elica Pach\'on\footnotemark[4]\ \footnotemark[6]}
\date{}

\begin{document}

\maketitle
\footnotetext{First Published in SIAM J. DISCRETE MATH. in Vol.~29, No.~2, pp.~1042–1064, published by the Society for Industrial and Applied Mathematics (SIAM). \copyright\ 2015 Society for Industrial and Applied Mathematics.}
\footnotetext[1]{Institute of Optimization and Discrete Mathematics, Graz University of Technology, Steyrergasse 30, 8010 Graz, Austria. E-mail: \{kang,ckoch\}@math.tugraz.at}
\footnotetext[4]{Department of Mathematics "Giuseppe Peano", University of Turin, Via Carlo Alberto 10, 10123 Turin, Italy. E-mail: angelicayohana.pachonpinzon@unito.it}
\footnotetext[2]{Supported by the Austrian Science Fund (FWF): P26826 and W1230, and the German Research Foundation (DFG): KA 2748/3-1.}
\footnotetext[3]{Supported by the Austrian Science Fund (FWF): P26826 and W1230, and NAWI-Graz. }
\footnotetext[6]{Supported by the German Research Foundation (DFG): KA 2748/3-1. }

\begin{abstract}
We determine the asymptotic size of the largest component in the $2$-type binomial random graph $G(\mathbf{n},P)$ near criticality using a refined branching process approach. In $G(\mathbf{n},P)$ every vertex has one of two types, the vector $\mathbf{n}$ describes the number of vertices of each type, and any edge $\{u,v\}$ is present independently with a probability that is given by an entry of the probability matrix $P$ according to the types of $u$ and $v.$

We prove that in the \emph{weakly} supercritical regime, i.e.\ if the \enquote{distance} to the critical point of the phase transition is given by an $\varepsilon=\varepsilon(\mathbf{n})\to0,$ with probability $1-o(1),$ the largest component in $G(\mathbf{n},P)$ contains asymptotically $2\varepsilon \|\mathbf{n}\|_1$ vertices and all other components are of size $o(\varepsilon \|\mathbf{n}\|_1).$  
\end{abstract}

\section{Introduction}

The theory of random graphs was founded by Erd\H{o}s and R\'enyi in the late 1950s. One of their most striking results concerned the \emph{phase transition} of the size of the largest component -- adding a few additional edges to a random graph can drastically alter the size of its largest component. In~\cite{ErdosRenyi60} they considered the random graph $G(n,m)$ obtained by choosing a graph uniformly at random amongst all graphs on $n$ (labelled) vertices containing precisely $m$ edges and proved the following result: Let $c\ge 0$ be any constant. If $c<1$, then \emph{with high probability} (\emph{whp} for short, meaning with probability tending to one as $n\to\infty$)  all components in $G(n,cn/2)$ have size $O(\log n)$, while if $c=1$, \emph{whp} the largest component is of size $\Theta(n^{2/3})$, and if $c > 1$, then \emph{whp} there is a component of size $\Theta(n),$ called  the \enquote{giant component}, and all other components are of size $O(\log n)$.

Bollob\'as~\cite{Bollobas84} investigated this phenomenon further and described in detail the behaviour of $G(n,m)$ when $m$ is close to $n/2$, i.e. $m=(1\pm\varepsilon)n/2$ for some $\varepsilon=\varepsilon(n)>0$ satisfying $\varepsilon \rightarrow 0$ as $n\to\infty$. His initial results were then improved by {\L}uczak~\cite{Luczak90}. In particular, if in addition $\varepsilon^3 n\to\infty$, \emph{whp} the largest component in $G(n,(1-\varepsilon)n/2)$ has size $o(n^{2/3})$, whereas the largest component in $G(n,(1+\varepsilon)n/2)$ contains asymptotically $2\varepsilon n$ vertices and all other components are of size $o(\varepsilon n)$. For a comprehensive account of the results see~\cite{AlonSpencerBook,BollobasBook,JansonLuczakRucinskiBook}. 

In the meantime many of these results have been reproved and strengthened using various modern techniques such as martingales~\cite{NachmiasPeres10a}, partial differential equations~\cite{SpencerWormald07}, and search algorithms~\cite{BollobasJansonRiordan07,KrivelevichSudakov13}. Furthermore, more complicated discrete structures like random hypergraphs have been studied~\cite{BehrischCojaOglanKang10, BollobasRiordan12c, KaronskiLuczak02}. 

Over the last years, random graphs have proved to have wide-ranging applications in neurobiology, statistical physics, and the modelling of complex networks~\cite{NewmanBook,Strogatz01}. Frequently some properties of real-world networks are already empirically \enquote{known} and have motivated the definition of more sophisticated random graph models~\cite{ChungLu02,ChungLu06a,Soderberg02}. In particular, applicable random graph models should allow for different types of vertices having different degree distributions, i.e.\ some level of \emph{inhomogeneity}.
 A general theory of inhomogeneous random graphs was developed by Bollob\'as, Janson, and Riordan~\cite{BollobasJansonRiordan07} providing a unified framework for a large number of previously studied random graph models~\cite{BollobasJansonRiordan05,BrittonDeijfenMartinLof06,NorrosReittu06}. For example they analysed the degree distribution, the number of paths and cycles, and the phase transition for the giant component.  The behaviour at the \emph{critical point} (corresponding to $G(n,cn/2)$ for $c=1$) has been studied by van der Hofstad in the so-called \emph{rank one} case~\cite{VanDerHofstad12}. Recently Bhamidi, Broutin, Sen, and Wang studied the general inhomogeneous random graph with a bounded number of types inside the \emph{critical window} (corresponding to $G(n,(1\pm\varepsilon)n/2)$ for some $\varepsilon=\varepsilon(n)>0$ satisfying $\varepsilon^3n\to C$,  $0\le C<\infty$) and have described the joint distribution of the largest components using Brownian motion~\cite{BhamidiBroutinSenWang15}.

In this paper, we study an inhomogeneous random graph model in which there are $n$ vertices, each vertex has one of two types, and an edge between a pair of vertices of types $i$ and $j$ is present with probability $p_{i,j}$ independently of all other pairs. The focus lies on the \emph{weakly} supercritical regime, i.e.\ when the distance to the critical point of phase transition decreases to zero as $n\to\infty$. In this regime the behaviour of the random graph depends very sensitively on the parameters and could not be studied using the parametrisation in~\cite{BollobasJansonRiordan07}.  We determine the size of the largest component in this regime (Theorem 2.1).

In order to derive the main results, we apply a simple breadth-first search approach to construct a rooted spanning tree of a component and couple it with a multi-type branching process with binomial offspring distributions, which is viewed as a random rooted tree. In addition, the width and the dual of that random rooted tree play important roles in the  second moment analysis.

The results of this paper are indeed not surprising and the techniques used in the paper may look familiar. The main contribution of this paper is that it shows how a simple branching process approach combined with the concepts of tree width and dual processes can be  applied nicely to a multi-type random graph \emph{all the way through} the supercritical regime. 

\section{Model and main results}

In this section we will first define multi-type binomial random graphs and associate them with branching processes. Then we state the main results and afterwards provide an outline of the proof and the methods involved. We conclude the section by discussing related results on the general inhomogeneous random graph studied in~\cite{BollobasJansonRiordan07}.

\subsection{Multi-type binomial random graph model}

Let $k\in\N$ be fixed. Every vertex is associated with a \emph{type} $i\in\{1,\dots,k\}$  and we denote by $V_i$ the set of all vertices of type $i\in\{1,\dots,k\}$. Given an arbitrary vector $\mathbf{n}=(n_1,\dots,n_k)\in\N^k$ and a symmetric matrix of probabilities $P=(p_{i,j})_{i,j=1,\dots,k}\in [0,1]^{k\times k}$ we consider the \emph{$k$-type binomial random graph} $G_k(\mathbf{n},P)$ on $n_l$ vertices of type $l$, for $l\in\{1,\dots.k\}$, with the following edge set: For each pair $\{u,v\}$, where $u$ is of type $i$ and $v$ of type $j$, we include the edge $\{u,v\}$ independently of any other pair with probability $p_{i,j}$ and exclude it with probability $1-p_{i,j}$. We  write $M=(\mu_{i,j})_{i,j\in\{1,2\}}$ for the matrix of the expected number of neighbours $\mu_{i,j}=p_{i,j}n_j$ of type $j\in\{1,\dots,k\}$ for a vertex of type $i\in\{1,\dots,k\}$.

 Next we associate a binomial branching process in which each individual has a type $i\in \{1,\dots,k\}$ with the random graph $G_k(\mathbf{n},P).$  Fix a time $t\in\N_0$ and let $I_t$ be a set of individuals (i.e. the population) at time $t,$ which we also call the $t$-th generation of individuals. Then, with each individual $v\in I_t$ of type $j'\in \{1,\dots, k\}$, we associate a random vector $\mathbf{X}^v=(X_{1}^v,\dots,X_{k}^v)$, where for each $j\in\{1,\dots,k\}$ the random variable  $X_{j}^v$ is independent and binomially distributed with parameters $n_{j}$ and $p_{j',j}$ and thus with mean $\mu_{j',j}$. Then the population $I_{t+1}$ at time $t+1$ will be a set containing exactly $\sum_{v\in I_t}{X^v_{j}}$ \emph{new} individuals of type $j$, for each $j\in\{1,\dots,k\}$. In other words, the random variable $X_{j}^v$ represents the number of children of type $j$ that are born from the individual $v$. A $k$-type binomial branching process starting with an initial population $I_0$ is a sequence of random vectors $(Z_{t}(1),\dots,Z_{t}(k))_{t\in\N_0}$ generated by iterating the construction described above, where $Z_{t}(j)$ is the random variable describing the number of individuals of type $j$ in the $t$-th generation for each $j\in\{1,\dots,k\}$ and $t\in \N_0$. For $i\in \{1,\dots,k\}$ we denote by $\mathcal{T}_{\mathbf{n},P}^i$ a $k$-type binomial branching process starting with a single vertex of type $i$. We may also use $\mathcal{T}_{\mathbf{n},P}^i$ to denote the rooted (possibly infinite) tree created by an instance of the branching process. The context will always clarify the notation. Furthermore, if a statement is independent of the starting type we simply write $\mathcal{T}_{\mathbf{n},P}$\,, for instance, we refer to the matrix $M$ as \emph{offspring expectation matrix}  of the branching process $\mathcal{T}_{\mathbf{n},P}$\,.

Observe that for $k=1$ we obtain the classical binomial random graph $G(n,p)$ where $n=n_1$ and $p=p_{1,1}$ and the corresponding binomial branching process. 

Throughout the paper we focus on the case $k=2$ and for simplicity we write $G(\mathbf{n},P):=G_2(\mathbf{n},P)$. We denote by $n=n_1+n_2$ the total number of vertices in $G(\mathbf{n},P)$ and without loss of generality we assume that $n_1\geq n_2$. Furthermore, unless specified explicitly, all asymptotic statements are to be understood in terms of $n_1$ and $n_2$ being \emph{large enough} yet fixed and we use the notation $\min\{n_1,n_2\}=n_2\to\infty$ for this. Note that in general $\eta_{1,2}\neq\eta_{2,1}$ and  it is possible that $\eta_{2,1}/\eta_{1,2}\to\infty$, even though $p_{1,2}=p_{2,1}.$

Given a graph $\mathcal{G}$ with components $C_1,\dots,C_r$ ordered by size such that $|C_1|\geq |C_2|\geq\dots\geq |C_r|$ we denote by $\mathcal{L}_i(\mathcal{G})=C_i$ the $i$-th largest component of $\mathcal{G}$ and its size by $L_i(\mathcal{G})=|\mathcal{L}_i(\mathcal{G})|=|C_i|$,  for any $i\in\{1,\dots,r\}$,  and set $\mathcal{L}_i(\mathcal{G})=\emptyset$ and $L_i(\mathcal{G})=0$ if $i>r.$
Moreover, we will use the following standard notation to describe asymptotic statements: For any real functions $f=f(n_1,n_2)$ and $g=g(n_1,n_2)$ we write: $f=O(g)$ if $\exists c>0,n_0$ such that $|f(n_1,n_2)|\leq c|g(n_1,n_2)|$ for all $n_1\geq n_2\geq n_0$; $f=o(g)$ if $\forall c>0: \exists n_0$ such that $|f(n_1,n_2)|\leq c|g(n_1,n_2)|$ for all $n_1\geq n_2\geq n_0$; $f=\Omega(g)$ if $g=O(f)$; $f=\Theta(g)$ if $f=O(g)$ and $f=\Omega(g)$ and $f\sim g$ if $f-g=o(g)$.

\subsection{Main results}\label{sectionMainResult}

We show that $G(\mathbf{n},P)$ exhibits a phase transition in the size of the largest component. In particular,  we show that  in the \emph{weakly} supercritical regime there is a unique largest component containing asymptotically $2\varepsilon n$ vertices.  In fact, we prove a stronger result.

\begin{theorem}\label{mainresult}
For $n_1\in\N$ and $n_2\in \N$ with $n_1\geq n_2$, let $n=n_1+n_2$ and let $\varepsilon=\varepsilon(n_1,n_2)>0$ with $\varepsilon=o(1).$ Furthermore, let $$P=\left(p_{i,j}\right)_{i,j\in \{1,2\}}\in(0,1]^{2\times 2}$$ be a symmetric matrix of probabilities satisfying the following conditions:
 \begin{equation}\label{asymptoticalCondition}
\varepsilon^3n_2 \min\{1,\varepsilon^{-1}\mu_{2,1}\}\to \infty,
\end{equation}
\begin{equation}
\mu_{\iota,1}+\mu_{\iota,2}=1+\varepsilon+o(\varepsilon), \text{ for any }\iota\in\{1,2\}, \label{sumExpectation}
\end{equation}
where $\mu_{i,j}=p_{i,j}n_j$ for every pair $(i,j)\in \{1,2\}^2$.
Then, whp the following holds for all integers $r\ge 2$ and $i\in\{1,2\}\!:$  
$$
|\mathcal{L}_1\left(G(\mathbf{n},P)\right)\cap V_i|=(2+o(1))\varepsilon n_i \quad\text{ and } \quad|\mathcal{L}_r\left(G(\mathbf{n},P)\right)\cap V_i|=o(\varepsilon n_i); 
$$
therefore, in particular, 
\[L_1\left(G(\mathbf{n},P)\right)=(2+o(1))\varepsilon n\quad\text{ and } \quad L_r\left(G(\mathbf{n},P)\right)=o(\varepsilon n).\] 
\end{theorem}

\begin{remark}
Observe that, up to the term $\min\{1,\varepsilon^{-1}\mu_{2,1}\}$, Condition~\eqref{asymptoticalCondition} mirrors the condition $\varepsilon^3 n\to \infty$ that is necessary and sufficient for the existence of a unique largest component in $G(n,(1+\varepsilon)/n)$. In $G(n,(1+\varepsilon)/n)$ the average degree is $1+\varepsilon=\Theta(1)$ and therefore it does not influence the asymptotic statement in~\eqref{asymptoticalCondition}. In $G(\mathbf{n},P)$ however, $P$ can be such that we are close to criticality but the average number $\mu_{2,1}$ (respectively $\mu_{1,2}$) of neighbours of the opposite type for a given vertex is still $o(1)$. Roughly speaking, it is reasonable that if $\mu_{2,1}$ is \enquote{very small}, then the random graph $G(\mathbf{n},P)$ may have two largest components, one of each type, that coexist independently since the probability of adding any edge between them is negligible. In particular, this would happen in case probability $p_{1,2}$ was equal to zero and therefore $\mu_{1,2}=\mu_{2,1}=0.$

On the other hand, in the special case $n_2=o(n),$ the condition in~\eqref{asymptoticalCondition} differs by an additional factor of $n_2/n=o(1)$ from that in $G(n,(1+\varepsilon)/n)$. This factor is, for instance, necessary to show that the number of vertices of type $2$ in the largest components is concentrated around its mean. Therefore it is not avoidable with this method, even though it might not be optimal.
\end{remark}

\begin{remark}
The symmetry of $P$ simply reflects the fact that $G(\mathbf{n},P)$ is an \emph{undirected} random graph.
\end{remark}

Note that the parameter $\varepsilon>0$ describes the \emph{distance} to the critical point for the emergence of the giant component in a sense that we will explain now. 
Roughly speaking, for some time, the breadth-first exploration process of a component in $G(\mathbf{n},P)$ looks like a $2$-type binomial branching process $\mathcal{T}_{\mathbf{n},P}$\,. This can be described by a coupling of the two processes. If the branching process dies out its total population should be rather \enquote{small}. Thus, by the coupling, the explored component is also \enquote{small}. It is well-known that for a $2$-type binomial branching process the property of survival has a threshold and that the critical point is characterised by the Perron-Frobenius eigenvalue 
\begin{equation}\label{lambdaEx}
\lambda=\frac{\mu_{1,1}+\mu_{2,2}}{2}+\frac{1}{2}\sqrt{(\mu_{1,1}+\mu_{2,2})^2+4\left(\mu_{1,2}\mu_{2,1}-\mu_{1,1}\mu_{2,2}\right)}
\end{equation} of its offspring expectation matrix $M=(\mu_{i,j})_{i,j\in\{1,2\}}$.  If $\lambda>1,$ the process has a positive probability of survival, while if $\lambda\leq 1,$ it dies out with probability $1$.

Next, let us compute $\lambda$ for the $2$-type binomial branching process $\mathcal{T}_{\mathbf{n},P}$ with parameters as in Theorem~\ref{mainresult}. Condition~\eqref{sumExpectation} states that for every constant $\delta\in(0,1)$ there is an $n_0=n_0\left(\delta\right)$ such that we have $$\left|\mu_{i,1}+\mu_{i,2}-(1+\varepsilon)\right|\leq \delta\varepsilon\, ,$$ for $i\in\{1,2\}$ and all $n_1\geq n_2\geq n_0.$ This implies $$\mu_{1,2}\mu_{2,1}-\mu_{1,1}\mu_{2,2}\leq \left(1+\varepsilon+\delta\varepsilon\right)^2-\left(\mu_{1,1}+\mu_{2,2}\right)\left(1+\varepsilon+\delta\varepsilon\right)$$
and similarly $$\mu_{1,2}\mu_{2,1}-\mu_{1,1}\mu_{2,2}\geq \left(1+\varepsilon-\delta\varepsilon\right)^2-\left(\mu_{1,1}+\mu_{2,2}\right)\left(1+\varepsilon-\delta\varepsilon\right).$$
 Therefore we can bound the argument of the square root in~\eqref{lambdaEx} from above by
$$(\mu_{1,1}+\mu_{2,2})^2+4\left(\mu_{1,2}\mu_{2,1}-\mu_{1,1}\mu_{2,2}\right)\leq \left(2\left(1+\varepsilon+\delta\varepsilon\right)-\left(\mu_{1,1}+\mu_{2,2}\right)\right)^2 $$ and from below by 
$$(\mu_{1,1}+\mu_{2,2})^2+4\left(\mu_{1,2}\mu_{2,1}-\mu_{1,1}\mu_{2,2}\right)\geq \left(2\left(1+\varepsilon-\delta\varepsilon\right)-\left(\mu_{1,1}+\mu_{2,2}\right)\right)^2\, . $$ 
Thus, by~\eqref{lambdaEx} and since $\delta$ was arbitrary, we obtain the following asymptotic estimate for the Perron-Frobenius eigenvalue
\begin{equation}\label{lambda}
\lambda=1+\varepsilon+o(\varepsilon).\end{equation}  
In other words, $\varepsilon$ describes how close $\lambda$ is to $1$.

Our next result concerns the weakly \emph{subcritical} regime.

\begin{theorem}\label{mainresult2}
For $n_1\in\N$ and $n_2\in \N$ with $n_1\geq n_2$, let $n=n_1+n_2$ and let $\varepsilon=\varepsilon(n_1,n_2)>0$ with $\varepsilon=o(1).$ Furthermore, let $$P=\left(p_{i,j}\right)_{i,j\in \{1,2\}}\in(0,1]^{2\times 2}$$ be a symmetric matrix of probabilities satisfying the following conditions:
 \begin{equation}\label{asymptoticalConditionSub}
\varepsilon^3n_2\to \infty,
\end{equation}
\begin{equation}
\mu_{\iota,1}+\mu_{\iota,2}=1-\varepsilon+o(\varepsilon), \text{ for any }\iota\in\{1,2\}, \label{sumExpectationSub}
\end{equation}
where $\mu_{i,j}=p_{i,j}n_j$ for every pair $(i,j)\in \{1,2\}^2$. Then we have whp
\begin{equation*}
L_1\left(G(\mathbf{n},P)\right)=o(n^{2/3}).
\end{equation*}
\end{theorem}

Note that analogously to Theorem~\ref{mainresult} the parameter $\varepsilon>0$ describes the \emph{distance} to the critical point from below. In other words, by~\eqref{sumExpectationSub}, we know that the Perron-Frobenius eigenvalue $\lambda$ of the offspring expectation matrix $M$ of a $2$-type binomial branching process $\mathcal{T}_{\mathbf{n},P}$ with parameters as in Theorem~\ref{mainresult2} satisfies
\begin{equation}\label{lambdaSub}
\lambda=1-\varepsilon+o(\varepsilon).
\end{equation} 

We will dedicate most of this paper  to the more sophisticated weakly supercritical regime. A sketch of the proof of Theorem~\ref{mainresult} is given in Subsection~\ref{sketch}, properties of supercritical branching processes will be analysed in Section~\ref{branchingSection}, and the actual proof of Theorem~\ref{mainresult} is provided in Section~\ref{proof}. The weakly subcritical regime follows in  Section~\ref{sec:weaklysub} with the proof of Theorem~\ref{mainresult2}. In Sections~\ref{sec:supercritical}~and~\ref{sec:subcritical} we also consider the size of the largest component in the regimes where the distance to the critical value is a \emph{constant} (independent of $n_1$ and $n_2$). In the supercritical regime the largest component will already be a \emph{giant component}, i.e.\ it is unique and of linear size. Similarly, we also get a stronger upper bound on the size of all components in the subcritical regime. These results can also be proved using the general framework in~\cite{BollobasJansonRiordan07}, however, we give alternative simple proofs.

\subsection{Proof sketch of Theorem~\ref{mainresult}}\label{sketch}

We extend the method employed by Bollob\'as and Riordan~\cite{BollobasRiordan12} for the study of the weakly supercritical regime of $G(n,p)$. To prove Theorem~\ref{mainresult} we consider the set $S$ of vertices in \enquote{large} components. The first goal is to show that the size of $S$ is concentrated around $2\varepsilon n$ by applying Chebyshev's inequality. We calculate asymptotically matching upper and lower bounds for the expected size of $S$ by coupling the breadth-first component exploration process from below and above with $2$-type branching processes. Once this is done, using a more refined version of this idea, we show that the square of this expectation is an upper bound for the second moment of the size of $S$, therefore the variance of the size of $S$ is indeed \enquote{small} compared to the square of the expectation and concentration follows by Chebyshev's inequality. So now we know that $whp$ the appropriate number of vertices lie in \enquote{large} components, but there might be several distinct such components all of which may also be much smaller than claimed in Theorem~\ref{mainresult}. However, we can construct a random graph via a two-round exposure. In the first round we reduce the probability of including some edges by a tiny bit and note that the above arguments will still hold in this setting. In the second round we once again look at each pair not yet connected by an edge and \enquote{sprinkle} an edge with a tiny probability independently for each such pair. By choosing the magnitude of these probabilities appropriately we can ensure that the resulting random graph has the same distribution as $G(\mathbf{n},P)$ and thus we can identify both random graphs by a coupling argument. Analysing the probability that \enquote{large} components are connected by at least one edge and using the union bound we show that $whp$ almost all vertices from $S$ lie in a single component of $G(\mathbf{n},P)$.

\subsection{Related work}\label{related}

The general inhomogeneous random graph model $G(n,c\kappa_n)$ studied by  Bollob\'as, Janson, and Riordan~\cite{BollobasJansonRiordan07} is closely related to the model $G(\mathbf{n},P)$. For any $n\in\N$ consider a random sequence ${\bf x}_n = (x_1,\dots, x_n)$ of points from a separable metric space ${\mathcal S}$ equipped with a Borel probability measure $\nu$ and let $\nu_n$ be the empirical distribution of  ${\bf x}_n$. Assume that $\nu_n$ converges in probability to $\nu$, then the triple $({\mathcal S},\nu,({\bf x}_n)_{n\geq 1})$ is called a \emph{vertex space}. Furthermore let $\{\kappa_n\}$ be a sequence of symmetric non-negative $\nu$-measurable functions  on ${\mathcal S}\times{\mathcal S}$, which converges to a limit $\kappa$, and let $c>0$ be a constant. Then the random graph $G(n,c\kappa_n)$ is a graph with vertex set $[n],$ where each pair of vertices $\{k,l\}$ is connected by an edge with probability $p_{k,l}:=\min\{1,c\kappa_n(x_k,x_l)/n\}$ independently of all other pairs.  

It is proved  that with respect to the parameter $c$ there is a phase transition concerning the size of the largest component. In particular, the existence and uniqueness of the giant component in $G(n,c\kappa_n)$ in the supercritical regime are proved using an appropriate multi-type branching process and analysing an integral operator $T_{\kappa}$. The critical point of the phase transition is characterised by $c_0:=||T_{\kappa}||^{-1}$: if $c\leq c_0,$ then the random graph contains only small components, but if $c>c_0,$ then there is a giant component which contains asymptotically $\rho_c n$ vertices,  where $\rho_c$ is independent of $n$ and grows linearly in $c-c_0>0.$ 

By contrast the focus of our paper lies on the \emph{weakly} supercritical regime (Theorem~\ref{mainresult}), i.e. the distance $\varepsilon=\varepsilon(\mathbf{n})$ from the critical point of the phase transition tends to zero as the number of vertices increases. The analysis in this regime is in general quite sophisticated, in comparison with the supercritical regime, i.e.\ when $\varepsilon>0$ is a constant independent of $\mathbf{n}.$ In general it is not sufficient to only scale the edge probabilities multiplicatively as in $G(n,c\kappa_n)$, since even if $\varepsilon\to0$ the \emph{spectral gap} of the operator $(1+\varepsilon)c_0T_\kappa$ is always bounded away from $0$. In contrast to this, the spectral gap of the offspring expectation matrix in $G(\mathbf{n},P)$ is given by $\mu_{1,2}+\mu_{2,1}$ and thus may\emph{ tend to zero arbitrarily quickly}. Similarly, if one of the types has significantly fewer vertices than the others, it will not influence the behaviour of $G(n,c\kappa_n)$; however we show that in the weakly supercritical regime of $G(\mathbf{n},P)$ these vertices may still be crucial in the evolution of the largest component and ignoring them may even result in a subcritical process. For $\varepsilon\to 0,$ such an example is given by $n_2=\sqrt{\varepsilon} n_1\, ,$ $\mu_{2,1}=1$ and thus $\mu_{1,1}=1-\sqrt{\varepsilon}+\varepsilon+o(\varepsilon).$

\section{Multi-type binomial branching processes in the supercritical regime}\label{branchingSection}

Later we will study the component sizes of the random graph $G(\mathbf{n},P)$ using $2$-type binomial branching processes. In this section we investigate some of their most important properties. We start with a simplified version of a key result concerning the survival probability of a general multi-type Galton-Watson branching processes.

\begin{lemma}[e.g.~\cite{HarrisBook}]\label{harris}
Let $\mathcal{T}_{\mathbf{n},P}$ be a $2$-type binomial branching process  with parameters $n_1\in\N$ and $n_2\in \N,$ with $n_1\geq n_2,$ and $$P=\left(p_{i,j}\right)_{i,j\in\{1,2\}}\in (0,1]^{2\times2}\, .$$  Let $\lambda=\lambda(n_1,n_2)>0$ be the Perron-Frobenius eigenvalue of its offspring expectation matrix $M=\left(\mu_{i,j}\right)_{i,j\in\{1,2\}},$ where $\mu_{i,j}=p_{i,j}n_j$\,, and let $(\rho_1,\rho_2)$ be the pair of survival probabilities. Then the following holds:
\begin{itemize}
\item if $\lambda\leq 1$, we have $\rho_1=\rho_2=0;$
\item if $\lambda>1$, then $(\rho_1,\rho_2)$ is the unique positive solution of 
\begin{equation}\label{PGF}
F_1(\rho_1,\rho_2)=F_2(\rho_1,\rho_2)=0,
\end{equation}
where 
\begin{equation}\label{PGF2}
F_i(\rho_1,\rho_2):=1-\rho_i-\left(1-\frac{\mu_{i,1}\rho_1}{n_1}\right)^{n_1}\left(1-\frac{\mu_{i,2}\rho_2}{n_2}\right)^{n_2},\text{ for } i\in\{1,2\}.
\end{equation}
\end{itemize}  
\end{lemma}

We call a branching process that has a positive survival probability \emph{supercritical} and otherwise we call it \emph{subcritical}. 

\begin{remark}
There is a very simple way to see that the survival probabilities must satisfy these equations: We consider the extinction probabilities before and after the \emph{first} step of the process and apply the Binomial Theorem.
\end{remark}

Because the conditions of Theorem~\ref{mainresult} imply that the Perron-Frobenius eigenvalue of the offspring expectation matrix $M$ is strictly larger than $1$, the associated branching process will have a positive survival probability that is given implicitly by~\eqref{PGF}. It is sufficient for us to extract some information about the asymptotic behaviour of the unique positive solution from these equations. However, even trying to solve these equations only asymptotically we have to be very careful with cancellation and take into account higher order terms: This is a major reason why the weakly supercritical regime is significantly harder to analyse than the other regimes.  

\subsection{Asymptotic survival probability}\label{survivalSection}

\begin{lemma}\label{asympSurvival}
Under the conditions as in Theorem~\ref{mainresult} the survival probabilities of the $2$-type binomial branching process $\mathcal{T}_{\mathbf{n},P}$ satisfy \[\rho_1\sim\rho_2\sim 2\varepsilon.\]
\end{lemma}

\begin{proof}
The key idea is to find suitable bounding functions for the $F_i$'s defined in~\eqref{PGF2}, for which the asymptotic values of the zeros can be computed easily, and then to observe that these coincide for the upper and lower bound.

First observe the following fact: If $\rho_1\geq \rho_2$, we have $F_1(\rho_1,\rho_2)\leq F_1(\rho_1,\rho_1)$ and $F_2(\rho_1,\rho_2)\geq F_2(\rho_2,\rho_2)$; Analogously, if $\rho_1<\rho_2$,  then $F_2(\rho_1,\rho_2)< F_2(\rho_2,\rho_2)$ and $F_1(\rho_1,\rho_2)> F_1(\rho_1,\rho_1)$. Thus, without loss of generality due to the Subsubsequence~Principle~(e.g.~\cite{JansonLuczakRucinskiBook}), we assume $\rho_1\geq \rho_2$  and consider the bounding functions $F_i(\rho_i,\rho_i),$ for $i\in\{1,2\}$:
 \begin{align*}
F_i(\rho_i,\rho_i)&=1-\rho_i-\left(1-\frac{\mu_{i,1}\rho_i}{n_1}\right)^{n_1}\left(1-\frac{\mu_{i,2}\rho_i}{n_2}\right)^{n_2}\\&= 1-\rho_i-\exp\left(-(\mu_{i,1}+\mu_{i,2})\rho_i-O\left(\frac{\mu_{i,1}^2\rho_i^2}{n_1}+\frac{\mu_{i,2}^2\rho_i^2}{n_2}\right)\right),
\end{align*}
by the Taylor-expansion of the natural logarithm around $1$. Since $\mu_{i,1}\leq 2$ and $\mu_{i,2}\leq 2,$ by the conditions of Theorem~\ref{mainresult} and the fact that $\rho_i\leq 1$ (since it is a probability), we have 
\begin{align*}
F_i(\rho_i,\rho_i)&=1-\rho_i-\exp\left(-\left[\mu_{i,1}+\mu_{i,2} + O\left(n_2^{-1}\right)\right]\rho_i\right)\\&=1-\rho_i-\exp\left(-(1+\varepsilon_i)\rho_i\right),
\end{align*}
where $\varepsilon_i=\mu_{i,1}+\mu_{i,2}-1+O\left(n_2^{-1}\right)\sim\varepsilon$, by~\eqref{asymptoticalCondition}~and~\eqref{sumExpectation}. We define 
$$
f_i(\rho_i):=1-\rho_i-\exp\left(-(1+\varepsilon_i)\rho_i\right)
$$
and note that solving 
\[f_i(\rho_i^*)=0\] 
asymptotically is a well-known problem that turns up when calculating the asymptotic value of the survival probability for a single-type Poisson branching process. Using the Taylor-expansion of the natural logarithm we get \[\varepsilon_i=\frac{-\log(1-\rho_i^*)-\rho_i^*}{\rho_i^*}=\sum_{m=1}^{\infty}\frac{(\rho_i^*)^m}{m+1}.\]
Since the coefficients in this series are all positive and $\varepsilon_i\to 0$, this shows that $\rho_i\to0$ and thus 
\[\varepsilon_i=\frac{\rho_i^*}{2}+O((\rho_i^*)^2).\] 
Having established the asymptotic behaviour of $\rho_1^*$ and $\rho_2^*$ it remains to show that $\rho_2^*\leq \rho_2$ and $\rho_1\leq\rho_1^*$, since this together with $\rho_2\leq\rho_1$ and $\varepsilon_2\sim\varepsilon_1\sim\varepsilon$ implies $\rho_2\sim\rho_1\sim 2\varepsilon$. 

For this last step, assume towards contradiction that $\rho_1>\rho_1^*$ and observe that $f_1$ is negative on the interval $(\rho_1^*,1]$. Since $(\rho_1,\rho_2)$ is by definition a solution of~\eqref{PGF} we have
$$0=F_1(\rho_1,\rho_2)\leq f_1(\rho_1)<0,$$ 
a contradiction. Analogously, $\rho_2<\rho_2^*$ leads to a contradiction since $f_2$ is positive on $(0,\rho_2^*)$, completing the proof.
\end{proof}

\subsection{Dual processes}\label{dualSection}

In the proof of Theorem~\ref{mainresult} we consider the supercritical branching process $\mathcal{T}_{\mathbf{n},P}$ associated with $G(\mathbf{n},P)$ and we will need a good upper bound on the probability its total number of offspring of type $j\in\{1,2\}$ is at least $l_j$, for carefully chosen real functions $l_1$ and $l_2$.  Since this probability is $1$ if the process survives, this reduces to analysing the conditional probability given the event $\mathcal{D}$ that the process dies out. We call the resulting $2$-type binomial branching process the \emph{dual process} and we can describe its offspring distributions as follows. We need to know, for a vertex $v$ of type $i\in\{1,2\}$ born in generation $I_t$, for some integer $t\ge 0,$ and a potential child $u$ of type $j\in\{1,2\},$ whether the edge $e=\{u,v\}$ is present in the dual process, i.e.\ conditioned on $\mathcal{D}$. Let $\event{e}$ be the event that $u$ is a child of $v$ in $\mathcal{T}_{\mathbf{n},P}$ and note that conditioning on $\event{e}$ will decrease the probability of $\mathcal{D}.$ More precisely, let $\mathbf{Y}=(Y_1,Y_2)$ denote the vector of the number of individuals of each type in generation $I_{t+1}$. Since $Y_1$ and $Y_2$ are independent binomially distributed random variables, calculating $\Prob\left(\mathcal{D}| \event{e}\right)$ and $\Prob\left(\mathcal{D}|\neg \event{e}\right)$ by conditioning on $\mathbf{Y}$ leads to
\begin{align*}
\Prob\left(\mathcal{D}\cond \event{e}\right) &= \sum_{r_j=0}^{n_j-1}\Prob\left(Y_j=r_j+1\cond \event{e}\right)(1-\rho_j)^{r_j+1}\\
&\qquad\qquad\qquad\qquad\quad\cdot\sum_{r_{3-j}=0}^{n_{3-j}}\Prob\left(Y_{3-j}=r_{3-j}\cond \event{e}\right)(1-\rho_{3-j})^{r_{3-j}}
\end{align*}
and 
\begin{align*}
\Prob\left(\mathcal{D}\cond\neg \event{e}\right) & =\sum_{r_j=0}^{n_j-1}\Prob\left(Y_j=r_j\cond \neg\event{e}\right)(1-\rho_j)^{r_j}\\
&\qquad\qquad\qquad\qquad\quad\cdot\sum_{r_{3-j}=0}^{n_{3-j}}\Prob\left(Y_{3-j}=r_{3-j}\cond \neg\event{e}\right)(1-\rho_{3-j})^{r_{3-j}}.
\end{align*}
Observe that by definition
$$\Prob\left(Y_j=r_j+1\cond \event{e}\right)=\Prob\left(Y_j=r_j\cond \neg\event{e}\right),$$
for all $r_j=0,\dots,n_j-1,$ and $Y_{3-j}$ is independent of $\event{e}$ and thus 
\begin{align}
\frac{\Prob\left(\mathcal{D}\cond \event{e}\right)}{\Prob\left(\mathcal{D}\cond \neg \event{e}\right)}&=1-\rho_j.\label{ProbCondEdge2}
\end{align} Therefore we get 
\begin{align}
\Prob\left(\event{e}|\mathcal{D}\right)\nonumber&=\frac{\Prob\left(\mathcal{D}| \event{e} \right)\Prob(\event{e})}{\Prob\left(\mathcal{D}| \event{e}\right)\Prob(\event{e})+\Prob\left(\mathcal{D}|\neg \event{e} \right)\Prob(\neg \event{e})}\\\nonumber&=\frac{\frac{\Prob\left(\mathcal{D}\cond \event{e}\right)}{\Prob\left(\mathcal{D}\cond \neg \event{e}\right)}\cdot\Prob(\event{e})}{\frac{\Prob\left(\mathcal{D}\cond \event{e}\right)}{\Prob\left(\mathcal{D}\cond \neg \event{e}\right)}\cdot\Prob(\event{e})+\Prob(\neg \event{e})}\\\nonumber&\stackrel{\eqref{ProbCondEdge2}}{=}\frac{p_{i,j}(1-\rho_j)}{1-\rho_j p_{i,j}}=:\pi_{i,j}\, ,
\end{align} 
uniformly for all edges $e$ (with one end point of type $i$ and the other of type $j$). An analogous calculation shows that the presence of $e$ does not depend on any other edges, i.e.\ the dual process is also a $2$-type binomial branching process. Hence, we write $\Pi=(\pi_{i,j})_{i,j\in\{1,2\}}$, $H=(h_{i,j})_{i,j\in\{1,2\}}$, where $h_{i,j}:=\pi_{i,j}n_j$, and denote the dual process of $\mathcal{T}_{\mathbf{n},P}$ by $\mathcal{T}_{\mathbf{n},\Pi}$. 

Intuitively it is obvious that the dual process of any supercritical process is subcritical. For completeness we give a short proof for the processes that we use. First observe that for each pair $(i,j)\in\{1,2\}^2$ we have, by Condition~\eqref{sumExpectation}, $p_{i,j}=O(n_j^{-1})$ and thus 
\begin{equation}
\pi_{i,j}=p_{i,j}(1-\rho_j)(1+O(n_j^{-2}\rho_j)).\label{pi}
\end{equation} 

\begin{lemma}
Let $\mathcal{T}_{\mathbf{n},P}$ be a $2$-type binomial branching process satisfying the conditions of Theorem~\ref{mainresult}. Then the offspring expectation matrix $H=(h_{i,j})_{i,j\in\{1,2\}}$ of the dual process $\mathcal{T}_{\mathbf{n},\Pi}$ satisfies 
\begin{equation}\label{dualLargestEigenvalue}
h_{\iota,1}+h_{\iota,2}=1-\varepsilon+o(\varepsilon), \text{ for } \iota\in\{1,2\},
\end{equation} and thus we have  $\lambda=1-\varepsilon+o(\varepsilon)$ for the Perron-Frobenius eigenvalue $\lambda$ of $H$.
\end{lemma} 

\begin{proof}
By~\eqref{pi}, Lemma~\ref{asympSurvival} and Condition~\eqref{sumExpectation} we get $$h_{\iota,1}+h_{\iota,2}=\left(\mu_{\iota,1}+\mu_{\iota,2}\right)(1-2\varepsilon)=1-\varepsilon+o(\varepsilon), \text{ for } \iota\in\{1,2\}.$$  The second statement follows analogously to~\eqref{lambdaSub}.
\end{proof}

 The benefit of using the \emph{subcritical} dual process $\mathcal{T}_{\mathbf{n},\Pi}$ is that we can bound the expected total number of offspring of each type.

\begin{lemma}\label{dualProcess}
For $i\in\{1,2\}$ let $\mathcal{T}_{\mathbf{n},P}^i$  be a $2$-type binomial branching process satisfying the conditions of Theorem~\ref{mainresult}. Then the associated dual process $\mathcal{T}_{\mathbf{n},\Pi}^i$ satisfies 
\begin{equation}\label{dualExp}
\E\left(\left|\mathcal{T}_{\mathbf{n},\Pi}^i\cap V_j\right|\right)\leq\varepsilon^{-1}, \text{ for }j\in\{1,2\}.
\end{equation} 
Moreover, for any real functions $l_1$ and $l_2$\,, this implies that \[\Prob\left(\left|\mathcal{T}_{\mathbf{n},P}^i\cap V_1\right|\geq l_1\vee\left|\mathcal{T}_{\mathbf{n},P}^i\cap V_2\right|\geq l_2\right)\leq 2\varepsilon+\varepsilon^{-1}l_1^{-1}+\varepsilon^{-1} l_2^{-1}+o(\varepsilon),\] and in particular
\begin{equation}\label{dualProbEst}
\Prob\left(\left|\mathcal{T}_{\mathbf{n},P}^i\cap V_1\right|\geq l_1\vee\left|\mathcal{T}_{\mathbf{n},P}^i\cap V_2\right|\geq l_2\right)\leq (2+o(1))\varepsilon,
\end{equation} if $\varepsilon^2 l_1\to\infty$ and $\varepsilon^2 l_2\to\infty.$
\end{lemma}

\begin{proof}
Consider the dual process $\mathcal{T}_{\mathbf{n},\Pi}$. We associate a vertex born in generation $I_t$\,,  for  integer $t\geq 1,$ with its \enquote{line of ancestry}, i.e.\ the string $\sigma \in \Sigma_t:=\{1,2\}^{t+1}$ which is the finite sequence of types of all its ancestors (starting with the root of $\mathcal{T}_{\mathbf{n},\Pi}$ and including itself). Set $\Sigma:=\bigcup_{t\ge 1}\Sigma_t$ and denote by $\Xi^*$  the set of all finite strings over the alphabet $\Xi:=\big\{(1,1),(1,2),(2,1),(2,2)\big\}.$ We consider the injective function $f\colon\Sigma\to\Xi^*$ defined by 
$$
f\big|_{\Sigma_t}\colon \Sigma_t\to \Xi^*,\,\sigma \mapsto \Big(\big(\sigma(0),\sigma(1)\big),\big(\sigma(1),\sigma(2)\big),\dots,\big(\sigma(t-1),\sigma(t)\big)\Big),
$$
 for $t\ge 1$. A string $\tau\in\Xi^*$ is called \emph{admissible} if $\tau\in f(\Sigma)$ and we denote the set of admissible strings by $\Xi^{ad}:=f(\Sigma)$.  Observe that, for every pair $(i,j)\in\{1,2\}^2$, the function $f$ can be seen as a bijection that maps the subset $\Sigma_{i,j}\subset \Sigma$ of lines of ancestry starting with $i$ and ending in $j$ to the subset $\Xi^{ad}_{i,j}\subset\Xi^{ad}$ of admissible strings starting with $(i,1)$ or $(i,2)$ and ending with $(1,j)$ or $(2,j)$. Next we define a function 
 $$
 \tilde{g}\colon\Xi\to \R_{>0},\,(i,j)\mapsto h_{i,j}
 $$
 that canonically extends to a function 
 $$
 g\colon\Xi^*\to \R_{>0},\, \tau\mapsto \prod_{r=0,\dots,t-1}\tilde{g}(\tau(r))
 $$
 and note that the expected number of offspring with a fixed line of ancestry $\sigma\in\Sigma$ is precisely $g(f(\sigma)).$ Hence, for $i\in\{1,2\},$ we obtain
\begin{align*}
\E\left(\left|\mathcal{T}_{\mathbf{n},\Pi}^i\cap V_i\right|\right)&=1+\sum_{\tau\in\Xi^{ad}_{i,i}}g(\tau)=1+G_{i,i}\left(h_{1,1},h_{1,2},h_{2,1},h_{2,2}\right),\\
\E\left(\left|\mathcal{T}_{\mathbf{n},\Pi}^i\cap V_{3-i}\right|\right)&= \sum_{\tau\in\Xi^{ad}_{i,3-i}}g(\tau)=G_{i,3-i}\left(h_{1,1},h_{1,2},h_{2,1},h_{2,2}\right),
\end{align*} 
where, for every pair $(j,j')\in\{1,2\},$ $G_{j,j'}$ is the four-variate ordinary generating function of $\Xi^{ad}_{j,j'}$ (marking the occurrences of $(1,1)$, $(1,2)$, $(2,1)$, and $(2,2)$ respectively). Using the \emph{symbolic method}(e.g.~\cite{FlajoletSedgewickBook}) we compute the closed forms of these generating functions providing
\begin{align*}
\E\left(\left|\mathcal{T}_{\mathbf{n},\Pi}^i\cap V_i\right|\right)=\frac{1-h_{3-i,3-i}}{d}\qquad\text{and}\qquad\E\left(\left|\mathcal{T}_{\mathbf{n},\Pi}^i\cap V_{3-i}\right|\right)=\frac{h_{i,3-i}}{d},
\end{align*} 
for $i\in\{1,2\},$ if the denominator
$$
d=1-h_{1,1}-h_{2,2}+h_{1,1}h_{2,2}-h_{1,2}h_{2,1}
$$
 is positive. We now establish stronger lower bounds, which allow us to prove statement~\eqref{dualExp}. By~\eqref{dualLargestEigenvalue}, for any constant $\delta\in(0,1)$ there exists an integer $n_*=n_*(\delta)$ such that for all $n_1\geq n_2\geq n_*$ we have 
\begin{align*}
\nonumber d&\geq 1-h_{1,1}-h_{2,2}+h_{1,1}h_{2,2}-(1-\varepsilon(1-\delta)-h_{1,1})(1-\varepsilon(1-\delta)-h_{2,2})\\
\nonumber&=\varepsilon(1-\delta)\left(2-h_{1,1}-h_{2,2}-\varepsilon(1-\delta)\right)\\
\nonumber&\geq\varepsilon (1-\delta)(2-(1-\varepsilon(1-\delta))-h_{3-i,3-i}-\varepsilon(1-\delta))\\
&= \varepsilon(1-\delta)(1-h_{3-i,3-i})\\
&\geq\varepsilon(1-\delta)h_{i,3-i}>0,
\end{align*}
and thus, letting $\delta\to 0,$ we obtain
\begin{align*}
\E\left(\left|\mathcal{T}_{\mathbf{n},\Pi}^i\cap V_i\right|\right)&\leq\varepsilon^{-1}\qquad\text{ and }\qquad
\E\left(\left|\mathcal{T}_{\mathbf{n},\Pi}^i\cap V_{3-i}\right|\right)\leq\varepsilon^{-1}.
\end{align*}

As mentioned before, conditioning on $\mathcal{D}_i$ (the event that the primal process $\mathcal{T}_{\mathbf{n},P}^i$ dies out) and applying Markov's inequality we get 
\begin{align*}
&\Prob{\left(\left|\mathcal{T}_{\mathbf{n},P}^i\cap V_1\right|\geq l_1\right.\vee\left.\left|\mathcal{T}_{\mathbf{n},P}^i\cap V_2\right|\geq l_2\right)}
\\&\hspace{1cm}=\Prob\left(\neg\mathcal{D}_i\right)\Prob{\left(\left|\mathcal{T}_{\mathbf{n},P}^i\cap V_1\right|\geq l_1\vee\left.\left|\mathcal{T}_{\mathbf{n},P}^i\cap V_2\right|\geq l_2\right|\neg\mathcal{D}_i\right)}
\\ &\hspace{1.3cm}+\Prob{\left(\mathcal{D}_i\right)\Prob\left(\left|\mathcal{T}_{\mathbf{n},P}^i\cap V_1\right|\geq l_1\vee\left.\left|\mathcal{T}_{\mathbf{n},P}^i\cap V_2\right|\geq l_2\right|\mathcal{D}_i\right)}
\\&\hspace{1cm}=\rho_i+\Prob(\mathcal{D}_i)\Prob{\left(\left|\mathcal{T}_{\mathbf{n},\Pi}^i\cap V_1\right|\geq l_1 \vee\left|\mathcal{T}_{\mathbf{n},\Pi}^i\cap V_2\right|\geq l_2\right)}
\\&\hspace{1cm}\leq\rho_i+\Prob{\left(\left|\mathcal{T}_{\mathbf{n},\Pi}^i\cap V_1\right|\geq l_1\right)}+\Prob{\left(\left|\mathcal{T}_{\mathbf{n},\Pi}^i\cap V_2\right|\geq l_2\right)}
\\&\hspace{1cm}\leq 2\varepsilon+\varepsilon^{-1}l_1^{-1}+\varepsilon^{-1} l_2^{-1}+o(\varepsilon),
\end{align*}
completing the proof.
\end{proof}

\subsection{Width of a tree}\label{widthSection}

The last tool that we need for the proof is the concept of the width of a rooted tree. The \emph{width} $w(\mathcal{T})$ of a rooted tree $\mathcal{T}$ is defined as the supremum of the sizes of all its generations. In this context we will interpret any branching process as a potentially infinite random rooted tree.

\begin{lemma}\label{largeWidth}
Let $\mathcal{T}_{\mathbf{n},P}$ be a $2$-type branching process satisfying the conditions of Theorem~\ref{mainresult} and denote by $\mathcal{D}$ the event that this process dies out. Then for any real function $m$ such that $\varepsilon m\to\infty$ we have \[\Prob\left(\{w\left(\mathcal{T}_{\mathbf{n},P}\right)\geq m\}\cap \mathcal{D}\right)=o(\varepsilon).\]
\end{lemma}

\begin{proof}
Denote $\mathcal{W}_m=\left\{w\left(\mathcal{T}_{\mathbf{n},P}\right)\geq m\right\}$ and let us construct $\mathcal{T}_{\mathbf{n},P}$ generation by generation and stop as soon as we see the first generation of size at least $m$ if there is one. Then we have $m_1$ vertices of type $1$ and $m_2$ vertices of type $2$ where $m_1+m_2\ge m$. Since each of the vertices of this generation starts an independent copy of $\mathcal{T}_{\mathbf{n},P}^1$ (respectively $\mathcal{T}_{\mathbf{n},P}^2$) we get for the probability of dying out given that $\mathcal{W}_m$ holds
\begin{equation*}
\Prob(\mathcal{D}|\mathcal{W}_m)=(1-\rho_1)^{m_1}(1-\rho_2)^{m_2}\leq e^{-(\rho_1 m_1+\rho_2 m_2)}=O(\exp(-2\varepsilon m)) =o(1),
\end{equation*}
where the asymptotic statements hold due to Lemma~\ref{asympSurvival} and since $\varepsilon m\to \infty.$
Hence, we obtain
 $$\Prob(\neg\mathcal{D}|\mathcal{W}_m)=1-o(1),$$ 
 and by the law of conditional probability and Lemma~\ref{asympSurvival} we have
\begin{align*}
\Prob\left(\mathcal{W}_m\wedge \mathcal{D}\right)=\Prob\left(\mathcal{W}_m\wedge\neg \mathcal{D} \right)\cdot\frac{\Prob(\mathcal{D}|\mathcal{W}_m)}{\Prob(\neg\mathcal{D}|\mathcal{W}_m)}
\leq o(\Prob(\neg\mathcal{D}))=o(\varepsilon),
\end{align*}
proving Lemma~\ref{largeWidth}.
\end{proof}

\section{Large components: proof of Theorem~\ref{mainresult}}\label{proof}

The main idea of the proof is to couple the component exploration process in $G(\mathbf{n},P)$ with instances of the $2$-type binomial branching process $\mathcal{T}_{\mathbf{n},P}$\,. Given a vertex $v$ of type $i$ in $G(\mathbf{n},P)$ we denote its component by $\mathcal{C}_v$. Furthermore let $\mathcal{T}_v$ be the random spanning-tree rooted at $v$ constructed by exploring new neighbours in $\mathcal{C}_v$ via a breadth-first search. Again we interpret branching processes as potentially infinite random rooted trees.
\begin{lemma}\label{coupling}
Given any vector $\mathbf{n}\in \N^2$, any symmetric matrix $P\in[0,1]^{2\times 2},$ and any vertex $v$ of type $i\in\{1,2\},$ the following two statements hold.
\begin{enumerate}[(i)]
\item There is a coupling of the random rooted trees $\mathcal{T}_v$ and $\mathcal{T}_{\mathbf{n},P}^i$ such that $\mathcal{T}_v\subset\mathcal{T}_{\mathbf{n},P}^i$. In particular, $\left|\mathcal{C}_v\cap V_j\right|\leq \left|\mathcal{T}_{\mathbf{n},P}^i\cap V_j\right|,$ for $j\in\{1,2\}.$
\item For any vector $\mathbf{m}=(m_1,m_2)\in \N^2$ satisfying $\mathbf{m}\leq \mathbf{n}$, there is a coupling of the random rooted trees $\mathcal{T}_v$ and $\mathcal{T}_{\mathbf{n}-\mathbf{m},P}^i$ such that $\mathcal{T}_{\mathbf{n}-\mathbf{m},P}^i\subset\mathcal{T}_v$ or both trees contain at least $m_1$ vertices of type $1$ or $m_2$ vertices of type $2$. In particular, either $\left|\mathcal{C}_v\cap V_j\right|\geq \left|\mathcal{T}_{\mathbf{n}-\mathbf{m},P}^i\cap V_j\right|$, for $j\in\{1,2\}$, or the total number of vertices of type $r$ in $\mathcal{C}_v$ and $\mathcal{T}_{\mathbf{n}-\mathbf{m},P}^i$ is at least $m_r$ for some $r\in\{1,2\}.$
\end{enumerate}
\end{lemma}

\begin{proof}
For the first statement, we generate $\mathcal{T}_v$ and $\mathcal{T}_{\mathbf{n},P}^i$ simultaneously, restoring the set of potential neighbours in the breadth-first search by adding fictional vertices of the same type to the vertex set of $G(\mathbf{n},P)$ for each vertex already added as a neighbour. Any offspring of a fictional vertex is automatically fictional. In this way, for each type $j\in\{1,2\}$, we always have $n_j$ potential new neighbours of this type each chosen independently with probability $p_{j',j}$ according to the type $j'\in\{1,2\}$ of the current vertex. After removal of all fictional vertices from $\mathcal{T}_{\mathbf{n},P}^i$ we obtain $\mathcal{T}_v$. Therefore,  we have $\left(\mathcal{T}_v\cap V_j\right)\subset\left(\mathcal{T}_{\mathbf{n},P}^i\cap V_j\right)$  and since $\left|\mathcal{C}_v\cap V_j\right|=\left|\mathcal{T}_v\cap V_j\right|$ the first statement holds.

For the second statement we proceed as before with the slight change that in each step we choose for any type $j\in\{1,2\}$ exactly $n_j-m_j$ neighbours from all possible new neighbours of type $j$ and only add those independently with probability $p_{j',j},$ where $j'\in\{1,2\}$ is the type of the current vertex, and ignore all other vertices. Until we have encountered a total of at least $m_{r}$ vertices of type $r$ in $\mathcal{T}_{\mathbf{n}-\mathbf{m},P}^i$ for some $r\in\{1,2\}$ there are always enough vertices of each type to choose from. Assuming that this does not happen for any $r\in\{1,2\}$, we thus have $\left(\mathcal{T}_{\mathbf{n}-\mathbf{m},P}^i\cap 
V_j\right)\subset\left(\mathcal{T}_v\cap V_j\right)\subset\left(\mathcal{C}_v\cap V_j\right)$ and the claim follows.
\end{proof}

Using Lemma~\ref{dualProcess} and Lemma~\ref{coupling} we can now establish the expectation of the number of vertices in \enquote{large} components. For any type $i\in\{1,2\},$ we denote by $S_{i,\mathbf{L}}=S_{i,\mathbf{L}}\left(G(\mathbf{n},P)\right)$ the set of all vertices of type $i$ in components that contain at least $l_j$ vertices of type $j,$ for some $j\in\{1,2\}$ and a properly chosen pair  $\mathbf{L}=(l_1,l_2)$ of real functions. Moreover, we denote by $s_{i,\mathbf{L}}=\left|S_{i,\mathbf{L}}\right|$ the cardinality of this set.

\begin{lemma}\label{ExpectationUpperBound}
 Let $l_j$ be a real function satisfying $\varepsilon^2 l_j\to\infty,$ for $j\in\{1,2\}.$ Then \[\E\left(s_{i,\mathbf{L}}\right)\leq (2+o(1))\varepsilon n_i, \text{ for } i\in \{1,2\}.\]
\end{lemma}
\begin{proof}
For $i\in\{1,2\},$ by Lemma~\ref{coupling}(i) and linearity of expectation, we have \begin{align*}
\E\left(s_{i,\mathbf{L}}\right)&=\sum_{v\in V_i}\Prob\left(\left|\mathcal{C}_v\cap V_1\right|\geq l_1\vee\left|\mathcal{C}_v\cap V_2\right|\geq l_2\right)
\\&\leq n_i\Prob\left(\left|\mathcal{T}_{\mathbf{n},P}^i\cap V_1\right|\geq l_1\vee\left|\mathcal{T}_{\mathbf{n},P}^i\cap V_2\right|\geq l_2\right)\sim 2\varepsilon n_i,
\end{align*}
where the last step holds by equation~\eqref{dualProbEst} in Lemma~\ref{dualProcess}.
\end{proof}

\begin{lemma}\label{ExpectationLowerBound}
Let $l_j$ be a real function satisfying $l_j=o(\varepsilon n_j),$ for $j\!\in\!\{1,2\}.$ Then \[\E\left(s_{i,\mathbf{L}}\right)\geq (2+o(1))\varepsilon n_i, \text{ for } i\in \{1,2\}.\]
\end{lemma}
\begin{proof}
We apply Lemma~\ref{coupling}(ii) with $\mathbf{m}=\mathbf{L}=(l_1,l_2)$, since $l_j=o(\varepsilon n_j)$, for $j\in\{1,2\},$ and note that the parameters of the coupling branching process satisfy Conditions~\eqref{asymptoticalCondition}~and~\eqref{sumExpectation}. Hence, for $i\in\{1,2\},$  this yields by linearity of expectation
\begin{align*}
\E\left(s_{i,\mathbf{L}}\right)&=\sum_{v\in V_i}\Prob\left(\left|\mathcal{C}_v\cap V_1\right|\geq l_1\vee\left|\mathcal{C}_v\cap V_2\right|\geq l_2\right)
\\&\geq n_i\Prob\left(\left|\mathcal{T}_{\mathbf{n}-\mathbf{m},P}^i\cap V_1\right|\geq l_1\vee\left|\mathcal{T}_{\mathbf{n}-\mathbf{m},P}^i\cap V_2\right|\geq l_2\right)
\\&\geq n_i\Prob\left(\mathcal{T}_{\mathbf{n}-\mathbf{m},P}^i \text{ survives}\right)\sim 2\varepsilon n_i,
\end{align*}
where the last step holds due to Lemma~\ref{asympSurvival}.
\end{proof}

In the next lemma we will show that $s_{i,\mathbf{L}}\left(G(\mathbf{n},P)\right)$, i.e.\ the number of vertices of type $i$ in large components, is concentrated around its expectation.
\begin{lemma}\label{Concentration}
Let $l_j$ be a real function satisfying $\varepsilon^2 l_j\to\infty$ and $l_j=o(\varepsilon n_j)$, for $j\in\{1,2\}.$ Then whp \[s_{i,\mathbf{L}}\left(G(\mathbf{n},P)\right)=(2+o(1))\varepsilon n_i, \text{ for } i\in\{1,2\}.\]
\end{lemma}

\begin{proof}
Lemmas~\ref{ExpectationUpperBound}~and~\ref{ExpectationLowerBound} show that $\E(s_{i,\mathbf{L}})\sim 2\varepsilon n_i$, hence it is sufficient to derive the upper bound 
\begin{equation}\label{SecondMoment}
\E(s_{i,\mathbf{L}}^2)\leq (4+o(1))\varepsilon^2 n_i^2\sim \E\left(s_{i,\mathbf{L}}\right)^2\, \text{ for } i\in\{1,2\}.
\end{equation}
The reason for this is the classical \emph{second moment method} (e.g.~\cite{AlonSpencerBook,JansonLuczakRucinskiBook}): Equation~\eqref{SecondMoment} implies that for the random variable $s_{i,\mathbf{L}}$ the variance is of smaller order than the square of the expectation, i.e. $$\V\left(s_{i,\mathbf{L}}\right)=\E\left(s_{i,\mathbf{L}}^2\right)-\E\left(s_{i,\mathbf{L}}\right)^2\leq o\left(\E\left(s_{i,\mathbf{L}}\right)^2\right),\text{ for }i\in\{1,2\},$$ which provides concentration by Chebyshev's inequality.

Without loss of generality fix a type $i\in\{1,2\}$ for the rest of the proof. Furthermore, fix a vertex $v$ of type $i$ in $G(\mathbf{n},P)$. Once again we explore the component $\mathcal{C}_v$ of that vertex in a breadth-first search generating a tree $\mathcal{T}_v'\subset\mathcal{C}_v$. However, we will stop the exploration immediately, even midway through revealing the neighbours of one particular vertex, if one of the following two events occurs:
\begin{enumerate}[(i)]
\item we have already reached a total of $l_j$ vertices of type $j$ for some $j\in\{1,2\};$
\item there are $\varepsilon l_2$ vertices that have been reached (i.e.\ children of earlier vertices) but not yet fully explored (flipped a coin for each possible neighbour).
\end{enumerate}
Note that for stopping condition~(ii) we do not distinguish the types of vertices. Any vertex that has been reached but not fully explored is called \emph{boundary vertex}. Observe that this process will create at most $\varepsilon l_2+1\leq 2\varepsilon l_2$ boundary vertices. Furthermore, denote by $\mathcal{A}$ the event that the process stops due to~(i)~or~(ii), rather than because it has revealed the whole component $\mathcal{C}_v$. Note that 
\begin{equation}\label{A}
\left\{\left|\mathcal{C}_v\cap V_1\right|\geq l_1 \vee \left|\mathcal{C}_v\cap V_2\right|\geq l_2\right\} \implies \mathcal{A},
\end{equation} a fact that we will use later on. Now we estimate the probability that $\mathcal{A}$ holds: By the coupling in Lemma~\ref{coupling}(i) we may assume that $\mathcal{T}_v'\subset\mathcal{T}_v\subset \mathcal{T}_{\mathbf{n},P}^i$ and, since we proceed in a breadth-first manner, at every point of time all the boundary vertices are contained in at most two consecutive generations. Hence if $\mathcal{A}$ holds, either $\left|\mathcal{T}_{\mathbf{n},P}^i\cap V_j\right|\geq l_j,$ for some $j\in\{1,2\},$ or the total number of offspring of the process $\mathcal{T}_{\mathbf{n},P}^i$ is finite and $w(\mathcal{T}_{\mathbf{n},P}^i)\geq \varepsilon l_2/2$. As calculated in Lemma~\ref{dualProcess} the probability that the first case occurs is asymptotically at most $2\varepsilon,$ while for the second case we calculated in Lemma~\ref{largeWidth} that the probability of having large width but still dying out is $o(\varepsilon)$, hence 
\begin{equation}\label{AProb}
\Prob(\mathcal{A})\leq (2+o(1))\varepsilon.
\end{equation}

 We use this to relate the second moment to the expectation on the conditional probability space, where we condition on $\mathcal{A}$ holding. We replace $s_{i,\mathbf{L}}$ by a sum of indicator random variables and from the implication in~\eqref{A} we get
\begin{align}\label{conditionalSecondMoment}
\nonumber\E\left[s_{i,\mathbf{L}}^2\right]&=\sum_{v\in V_i}\E\left[\Ind_{\left\{\left|\mathcal{C}_v\cap V_1\right|\geq l_1\vee\left|\mathcal{C}_v\cap V_2\right|\geq l_2\right\}}s_{i,\mathbf{L}}\right]
\\&\leq n_i\E\left[\Ind_{\mathcal{A}}s_{i,\mathbf{L}}\right]\nonumber
\\&=n_i\Prob(\left.\mathcal{A})\E\left[s_{i,\mathbf{L}}\right|\mathcal{A}\right]\nonumber
\\&\leq (2+o(1))\varepsilon n_i\E\left[s_{i,\mathbf{L}}\left.\right|\mathcal{A}\right].
\end{align}

For the remainder of this proof we will compute an asymptotic upper bound for the conditional expectation $\E\left[s_{i,\mathbf{L}}\left.\right|\mathcal{A}\right].$ Now for any vertex $u\not\in \mathcal{T}_v'$ of type $i$ we reveal its component as before in a breadth-first manner but ignore any vertices that are in $\mathcal{T}_v'$, i.e.\ we explore in $G'=G(\mathbf{n},P)\backslash V(\mathcal{T}_v')$ until we have revealed the whole component in this subgraph. Moreover, we couple the generated tree $\mathcal{T}_u''$ with $\mathcal{T}_{\mathbf{n},P}^i$ such that $\mathcal{T}_u''\subset \mathcal{T}_{\mathbf{n},P}^i$. We denote by $\mathcal{D}_i$ the event that this instance of $\mathcal{T}_{\mathbf{n},P}^i$ dies out and note, in particular, that $\mathcal{D}_i$ is independent of the event $\mathcal{A}$, hence 
\begin{equation} \label{ProbSecExpl}
\Prob\left(\left.\neg \mathcal{D}_i\right|\mathcal{A}\right)=\Prob(\neg\mathcal{D}_i)=(2+o(1))\varepsilon
\end{equation} by Lemma~\ref{asympSurvival}. Let us observe that $\left|\mathcal{T}_u''\right|\leq \left|\mathcal{C}_u\right|$ and furthermore that equality holds unless $G(\mathbf{n},P)$ contains an edge connecting a boundary vertex to a vertex of $\mathcal{T}_u''$. Therefore, for any given $r\in\N,$ we have 
\begin{equation*}
\Prob{\left(\left|\mathcal{C}_u\right|\neq\left|\mathcal{T}_u''\right|\mid\mathcal{D}_i\wedge\mathcal{A}\wedge\left\{\left|\mathcal{T}_u''\right|=r \right\}\right)}\leq 2\varepsilon l_2 r \max\left\{p_{j,j'}\mid j,j'\in\{1,2\}\right\} ,
\end{equation*}
by the union bound, as there are at most $2\varepsilon l_2$ boundary vertices. Note that by~\eqref{sumExpectation} we have $$\max\left\{p_{j,j'}\mid j,j'\in\{1,2\}\right\}\leq (1+\varepsilon)n_2^{-1}\leq 2 n_2^{-1}.$$ Hence, by the law of total probability, 
\begin{align*}\label{boundaryHitProb}
\Prob{\left(\left|\mathcal{C}_u\right|\neq\left|\mathcal{T}_u''\right|\mid\mathcal{D}_i\wedge\mathcal{A}\right)}\leq 4\varepsilon l_2 n_2^{-1} {\E\left[\left|\mathcal{T}_u''\right| \mid\mathcal{D}_i\right]}.
\end{align*} In order to simplify notation we will write $$\mathcal{X}_{u,\mathbf{L}}=\left\{\left|\mathcal{C}_u\cap V_1\right|\geq l_1\vee\left|\mathcal{C}_u\cap V_2 \right|\geq l_2\right\}$$ for the event that the component of $u$ is large. Hence it follows that
\begin{align*}
\Prob{\left(\mathcal{X}_{u,\mathbf{L}}\mid\mathcal{A}\right)}\leq\Prob(\neg\mathcal{D}_i)&+\Prob(\mathcal{D}_i)\Prob{\left(\mathcal{X}_{u,\mathbf{L}}\mid\mathcal{D}_i\wedge\mathcal{A}\right)}
\\\leq\Prob(\neg\mathcal{D}_i)&+\Prob{\left(\mathcal{X}_{u,\mathbf{L}}\wedge \left\{\left|\mathcal{C}_u\right|=\left|\mathcal{T}_u''\right|\right\}\mid\mathcal{D}_i\wedge\mathcal{A}\right)}
\\&+\Prob{\left(\mathcal{X}_{u,\mathbf{L}}\wedge \left\{\left|\mathcal{C}_u\right|\neq\left|\mathcal{T}_u''\right|\right\}\mid\mathcal{D}_i\wedge\mathcal{A}\right)}
\\\leq\Prob(\neg\mathcal{D}_i)&+\Prob{\left(\left|\mathcal{T}_u''\cap V_1\right|\geq l_1\vee \left|\mathcal{T}_u''\cap V_2\right|\geq l_2\mid\mathcal{D}_i\right)}
\\&+\Prob{\left(\left|\mathcal{C}_u\right|\neq\left|\mathcal{T}_u''\right|\mid\mathcal{D}_i\wedge\mathcal{A}\right)}
\\\leq\Prob(\neg\mathcal{D}_i)&+\Prob{\left(\left|\mathcal{T}_{\mathbf{n},\Pi}^i\cap V_1\right|\geq l_1\right)}+\Prob{\left(\left|\mathcal{T}_{\mathbf{n},\Pi}^i\cap V_2\right|\geq l_2\right)}
\\&+4\varepsilon l_2 n_2^{-1} {\E\left[\left|\mathcal{T}_u''\right| \mid\mathcal{D}_i\right]}
\\\leq\Prob(\neg\mathcal{D}_i)&+l_1^{-1}\E{\left(\left|\mathcal{T}_{\mathbf{n},\Pi}^i\cap V_1\right|\right)}+l_2^{-1}\E{\left(\left|\mathcal{T}_{\mathbf{n},\Pi}^i\cap V_2\right|\right)}
\\&+4\varepsilon l_2 n_2^{-1}\E\left(\left|\mathcal{T}_{\mathbf{n},\Pi}^i\right|\right),
\end{align*}
where the last step holds due to Markov's inequality. Furthermore, we know that these expectations are all of order $O\left(\varepsilon^{-1}\right)$ by the bound~\eqref{dualExp} in Lemma~\ref{dualProcess}. Additionally, by our assumptions on $\mathbf{L}$, the coefficients $l_1^{-1},$ $l_2^{-1}$ and $4\varepsilon l_2 n_2^{-1}$ are all of order $o(\varepsilon)$ and therefore using~\eqref{ProbSecExpl} we get $$\Prob{\left(\mathcal{X}_{u,\mathbf{L}}\mid\mathcal{A}\right)}\leq(2+o(1))\varepsilon. $$
Thus, since there are at least $n_i-l_i$ vertices of type $i$ for which we can apply this bound, we get \[\E\left[s_{i,\mathbf{L}}\mid\mathcal{A}\right]\leq l_i+ (n_i-l_i)\Prob{\left(\mathcal{X}_{u,\mathbf{L}}\mid\mathcal{A}\right)}\leq l_i + (2+o(1))\varepsilon n_i=(2+o(1))\varepsilon n_i.\] Inserting this into inequality~\eqref{conditionalSecondMoment} and then applying Chebyshev's inequality completes the proof of Lemma~\ref{Concentration}.
\end{proof}

Now we can prove Theorem~\ref{mainresult}.

\begin{proof}[Proof of Theorem~\ref{mainresult}.]
Let us first introduce some further notation. We write 
\begin{align}
\nonumber\alpha&=\min\{1,\varepsilon^{-1}\mu_{2,1}\},\\
\omega&=\alpha\varepsilon^3 n_2\label{omega}
\end{align} and note that $\alpha>0$ and $\omega\to \infty$ by the assumptions of Theorem~\ref{mainresult}. We set 
\begin{equation}\label{asympL1}
l_j=\frac{\varepsilon n_j}{\log \omega}=o(\varepsilon n_j), \text{ for } j\in\{1,2\},
\end{equation} and note that we could replace $\log \omega$ by any function $\hat{\omega}$ such that $\hat{\omega}\to\infty$ but growing very slowly compared to $\omega$. Moreover, observe that since $\alpha\le 1$ we have
 \begin{equation}\label{asympL2}
\varepsilon^2 l_j=\frac{\varepsilon^3 n_j}{\log \omega}\ge\frac{\omega}{\log \omega}\to \infty, \text{ for } j\in\{1,2\}.
\end{equation} 

Essentially, we know so far that the random graph $G(\mathbf{n},P)$ satisfying the conditions of Theorem~\ref{mainresult}, contains the \enquote{right} number of vertices in large components. It only remains to show that all these components are connected, and thus form a single component, if we \enquote{sprinkle} some more edges. Formally this can be done as follows. We define a symmetric probability matrix $P^b$ by setting
$$
p_{1,2}^b=\frac{\alpha\varepsilon}{n_1\log \omega}=\min\left\{\frac{\varepsilon}{n_1\log \omega}, \frac{p_{1,2}}{\log \omega}\right\},
$$ 
and $p_{1,1}^b=p_{2,2}^b=0.$ Then let $P^a$ be the symmetric probability matrix whose entries satisfy 
$$
p_{1,2}^a+p_{1,2}^b-p_{1,2}^a p_{1,2}^b=p_{1,2}
$$
 and   
$$
p_{i,i}^a=p_{i,i}\text{, for } i\in\{1,2\}.
$$
We construct $G(n,P^a)$ and $G(n,P^b)$ independently and couple them in such a way that we have $$G(\mathbf{n},P^a)\cup G(\mathbf{n},P^b)= G(\mathbf{n},P).$$

Since $p_{1,2}^b\le p_{1,2}/\log\omega$ we have $p_{1,2}^a\ge p_{1,2}(1-1/\log\omega)$ implying that the entries of $P^a$ are all positive for large enough $n_1$ and $n_2\, .$ Furthermore, as $p_{1,2}^b=o( \varepsilon/n_1),$ we have 
$$
p_{1,2}^an_i=\mu_{3-i,i}+o(\varepsilon)\text{, for }i\in\{1,2\}
$$
 and therefore $G(n,P^a)$ also meets all requirements of Theorem~\ref{mainresult}. Moreover, we have calculated in~\eqref{asympL1}~ and~\eqref{asympL2} that the further conditions of Lemma~\ref{Concentration} are also satisfied for $\mathbf{L}=(l_1,l_2).$ Let us denote by  $S_{i,\mathbf{L}}^a=S_{i,\mathbf{L}}\left(G\left(\mathbf{n},P^a\right)\right)$ the set of vertices of type $i\in\{1,2\}$ in large components of $G(\mathbf{n},P^a)$, i.e. components containing at least $l_j$ vertices of type $j$ for some $j\in\{1,2\}.$ Then, by Lemma~\ref{Concentration}, we have $whp$ $$\left|S_{i,\mathbf{L}}^a\right|=2\varepsilon n_i +\zeta_i^a$$
  for some real function $\zeta_i^a=o(\varepsilon n_i).$ We assume that this event holds.
 
  Let $\mathcal{U}$ denote the set of all large components in $G(\mathbf{n},P^a)$. Then for any component $\mathcal{C}\in\mathcal{U}$ we say that the type $j\in\{1,2\}$ is a \emph{witness} for $\mathcal{C}$ being large if $\left|\mathcal{C}\cap V_j\right|>\frac{1}{2}l_j$. Observe that having a witness is a necessary condition for any component to be large, hence each large component $\mathcal{C}\in\mathcal{U}$ has at least one witness, yet it is not a sufficient condition. For $j\in\{1,2\}$ we define the set $\mathcal{U}^j\subset\mathcal{U}$ of large components for that type $j$ is a witness and write $\mathcal{U}^j=\left\{U_1^j,\dots,U_{r_j}^j \right\},$ for some integer $r_j\ge 0$. Intuitively, it should not be the case that one of these sets is empty. We prove this by a counting argument.

\begin{claim}\label{notStrange}
$\mathcal{U}^1$ and $\mathcal{U}^2$ are not empty, i.e. $r_1>0$ and $r_2>0.$
\end{claim}

\begin{proof}
Without loss of generality assume towards contradiction that $r_1=0,$ and thus clearly $r_2>0.$ Observe that this implies that 
$$\left|U_\iota^2\cap V_1\right|\leq \frac{1}{2}l_1=\frac{\varepsilon n_1}{2\log \omega},$$ and $$\left|U_\iota^2\cap V_2\right|\geq l_2=\frac{\varepsilon n_2}{\log \omega},
$$for $\iota\in\{1,\dots,r_2\}.$
Counting vertices of both types separately, we therefore get for type $1$ $$2\varepsilon n_1+\zeta_1^a=\left|S_{1,\mathbf{L}}^a\right|=\sum_{\iota=1}^{r_2}\left|U_\iota^2\cap V_1\right|\leq \frac{r_2\varepsilon n_1}{2\log \omega} $$
and for type $2$ 
$$2\varepsilon n_2+\zeta_2^a=\left|S_{2,\mathbf{L}}^a\right|=\sum_{\iota=1}^{r_2}\left|U_\iota^2\cap V_2\right|\geq \frac{r_2\varepsilon n_2}{\log \omega}.$$ This shows $$\left(4+\frac{2\zeta_1^a}{\varepsilon n_1}\right)\log\omega\leq r_2\leq \left(2+\frac{\zeta_2^a}{\varepsilon n_2}\right)\log \omega,$$ a contradiction for large enough $n_1$ and $n_2,$ since $\zeta_i^a\varepsilon^{-1}n_i^{-1}=o(1)$, for $i\in\{1,2\}.$ Hence Claim~\ref{notStrange} holds.
\end{proof}

We continue the proof of Theorem~\ref{mainresult}. Observe that by the definition of witnesses we have  $$\left|U_\iota^j\right|\geq\left|U_\iota^j\cap V_j\right|> \frac{1}{2}l_j=\frac{\varepsilon n_j}{2\log \omega},$$ for $j\in\{1,2\}$ and $\iota\in\{1,\dots,r_j\}.$ Hence, if we estimate the number of vertices of type $j$ by only summing over the components in $\mathcal{U}^j$ we get $$ \frac{r_j\varepsilon n_j}{2\log \omega}<\sum_{\iota=1}^{r_j}\left|U_\iota^j\cap V_j\right|\leq(2+o(1))\varepsilon n_j,$$ and consequently $$r_j\leq 5\log\omega,\text{ for }j\in\{1,2\}.$$ 
Let $U^1\in \mathcal{U}^1$ and $U^2\in \mathcal{U}^2$ be any two large components, i.e.\ they satisfy $\left|U^1\cap V_1\right|\ge \frac{1}{2}l_1=\frac{\varepsilon n_1}{2\log \omega}$ and $\left|U^2\cap V_2\right|\ge \frac{1}{2}l_2=\frac{\varepsilon n_2}{2\log \omega}.$ Then the probability that in $G(\mathbf{n},P^b)$ there is no edge between $U^1$ and $U^2$ is at most 
\begin{align*}
(1-p_{1,2}^b)^{\left|U^1\cap V_1\right|\left|U^2\cap V_2\right|}&\leq \exp\left(-\frac{\alpha\varepsilon }{n_1\log\omega}\cdot\frac{\varepsilon n_1}{2\log\omega}\cdot\frac{\varepsilon n_2}{2\log\omega}\right)\\
&\stackrel{\eqref{omega}}{=}\exp\left(-\frac{\omega} {4\log^3\omega}\right).
\end{align*}
 Taking the union bound for (up to) $r_1+r_2-1$ of these events shows that the probability that in $G(\mathbf{n},P)$ all components that were large in $G(\mathbf{n},P^a)$ are connected is at least $$1-(r_1+r_2-1)\exp\left(-\frac{\omega} {4\log^3\omega}\right)\geq 1-10\log \omega\exp\left(-\frac{\omega} {4\log^3\omega}\right)=1-o(1).$$
Thus, $whp$ there is a component $\mathcal{C}^*$ in $G(\mathbf{n},P^a)\cup G(\mathbf{n},P^b)=G(\mathbf{n},P)$ which  contains $S_{i,\mathbf{L}}^a,$ for $i\in\{1,2\}$. 

 On the other hand, writing $S_{i,\mathbf{L}}=S_{i,\mathbf{L}}\left(G(\mathbf{n},P)\right)$ for the set of vertices of type $i\in\{1,2\}$ in large components of $G(\mathbf{n},P)$ we get $(\mathcal{C}^*\cap V_i)\subset S_{i,\mathbf{L}}$ due to the coupling. Hence we have 
$$
S_{i,\mathbf{L}}^a\subset (\mathcal{C}^*\cap V_i)\subset S_{i,\mathbf{L}}.
$$
Furthermore, applying Lemma~\ref{Concentration}, with the same choice of $\mathbf{L},$ directly to $G(\mathbf{n},P)$ we obtain $whp$
$$\left|S_{i,\mathbf{L}}\right|=2\varepsilon n_i +\zeta_i$$
for some real function $\zeta_i=o(\varepsilon n_i).$
Thus the number of vertices of type $i$ in the component $\mathcal{C}^*$ satisfies
\begin{equation}\label{sizeC}
\Big|\left|\mathcal{C}^*\cap V_i\right|-2\varepsilon n_i\Big|\le \left|\zeta_i^a\right|+\left|\zeta_i\right|=o(\varepsilon n_i). 
\end{equation}

Moreover, any other large component $\mathcal{C}$ in $G(\mathbf{n},P)$ may at most contain all the vertices from $S_{1,\mathbf{L}}\setminus S_{1,\mathbf{L}}^a$ and $S_{2,\mathbf{L}}\setminus S_{2,\mathbf{L}}^a,$ and therefore satisfies
$$
\left|\mathcal{C}\cap V_i\right|\le \left|S_{i,\mathbf{L}}\right|-\left|S_{i,\mathbf{L}}^a\right|\le \left|\zeta_i\right|+\left|\zeta_i^a\right|=o(\varepsilon n_i), \text{ for } i\in\{1,2\}.
$$
In particular, summing over both types, we have 
$$
\left|\mathcal{C}\right|\le \left|\mathcal{C}^*\right|,
$$
for large enough $n_1$ and $n_2\,.$
Consequently, $\mathcal{C}^*$ is already the largest component $\mathcal{L}_1\left(G(\mathbf{n},P)\right)$ and satisfies the required asymptotics by~\eqref{sizeC}, completing the proof.
\end{proof}

\section{Weakly subcritical regime: proof of Theorem~\ref{mainresult2}}\label{sec:weaklysub}

Most of the work for this regime has already been done in Section~\ref{dualSection}, since the dual process of a weakly supercritical branching process is weakly subcritical. Therefore we will keep the proof short. 

\begin{proof}[Proof of Theorem~\ref{mainresult2}]
Let the conditions be as in Theorem~\ref{mainresult2}. Then, analogously to the proof of Lemma~\ref{dualProcess}, we calculate the expected total size of the $2$-type binomial branching process $\mathcal{T}_{\mathbf{n},P}^i$, for $i\in\{1,2\}$, and get
\[\E\left(\left|\mathcal{T}_{\mathbf{n},P}^i\right|\right)=\frac{1+\mu_{i,3-i}-\mu_{3-i,3-i}}{1-\left(\mu_{1,1}+\mu_{2,2}-\mu_{1,1}\mu_{2,2}+\mu_{1,2}\mu_{2,1}\right)}\sim \varepsilon^{-1}.\] 
Now let $L=\delta n^{2/3},$ for any fixed constant $\delta>0$ and write $S_L$ for the set of vertices in components of size at least $L$ and $s_l=\left|S_L\right|$. Then with the coupling as in Lemma~\ref{coupling}(i) we get, by applying Markov's inequality twice and linearity of expectation, 
\begin{align*}
\Prob\left(s_L\geq L\right)&\leq L^{-1}\E\left(s_L\right)\\&\leq L^{-1}\left(\sum_{v\in V_1}\Prob\left(\left|\mathcal{C}_v\right|\geq L\right)+\sum_{v\in V_2}\Prob\left(\left|\mathcal{C}_v\right|\geq L\right)\right)\\&\leq L^{-1}\left( n_1\Prob\left(\left|\mathcal{T}_{\mathbf{n},P}^1\right|\geq L\right)+n_2\Prob\left(\left|\mathcal{T}_{\mathbf{n},P}^2\right|\geq L\right)\right)\\&\leq \varepsilon^{-1}L^{-2}n=(\delta^{2}\varepsilon n^{1/3})^{-1}\to 0,
\end{align*} since $\varepsilon^3 n\to\infty$ by Condition~\eqref{asymptoticalConditionSub}. Hence, since $\delta>0$ was arbitrary, $whp$ all components are of size $o(n^{2/3})$.
\end{proof}

\begin{remark}
 This result can be slightly strengthened: Let $\omega=\varepsilon^3n\to \infty$ and replace $L$ by $\hat{L}=\delta n^{2/3}\omega^{-1/6+c}$ for any $0<c<1/6.$  
\end{remark}

\section{Supercritical regime}\label{sec:supercritical}

In the supercritical regime, when the distance from the critical point is a constant, $G(\mathbf{n},P)$ $whp$ has a giant component. The proof is essentially the same as in Section~\ref{proof} except for some of the arguments used for calculating the survival probabilities. 

\begin{theorem}\label{mainresult3}
For $n_1\in\N$ and $n_2\in \N$ with $n_1\geq n_2$, let $n=n_1+n_2$ and let $\varepsilon>0$ be a fixed constant. Furthermore, let $$P=\left(p_{i,j}\right)_{i,j\in \{1,2\}}\in(0,1]^{2\times 2}$$ be a symmetric matrix of probabilities satisfying the following conditions:
 \begin{equation}\label{asymptoticalConditionConst}
n_2 \mu_{2,1}\to \infty,
\end{equation}
\begin{equation}
\mu_{\iota,1}+\mu_{\iota,2}=1+\varepsilon+o(1), \text{ for any }\iota\in\{1,2\}, \label{sumExpectationConst}
\end{equation}
where $\mu_{i,j}=p_{i,j}n_j$ for every pair $(i,j)\in \{1,2\}^2$. Let $\rho_\varepsilon$ be the unique positive solution of the equation $$1-\rho_{\varepsilon}-\exp\left(-(1+\varepsilon)\rho_\varepsilon\right)=0.$$
Then, whp the following holds for every integer $r\ge 2$ and $i\in\{1,2\}\!:$ 
\[|\mathcal{L}_1\left(G(\mathbf{n},P)\right)\cap V_i|=\left(\rho_{\varepsilon}+o(1)\right)n_i\quad\text{ and } \quad|\mathcal{L}_r\left(G(\mathbf{n},P)\right)\cap V_i|=o(n_i). \] 
Therefore, in particular, 
 \[L_1\left(G(\mathbf{n},P)\right)=(\rho_{\varepsilon}+o(1)) n\quad\text{ and } \quad L_r\left(G(\mathbf{n},P)\right)=o(n).\]
\end{theorem}

\begin{proof}
Let the conditions be as in Theorem~\ref{mainresult3}. We will only show the computation for the survival probabilities, which is very similar to the proof of Lemma~\ref{asympSurvival}. For the $F_i$'s defined in~\eqref{PGF2} we use the same bounding functions as before.

As in the proof of Lemma~\ref{asympSurvival} we assume without loss of generality that $\rho_1\ge \rho_2$ and thus we have $F_2(\rho_1,\rho_2)< F_2(\rho_2,\rho_2)$ and $F_1(\rho_1,\rho_2)> F_1(\rho_1,\rho_1)$. Fix $i\in\{1,2\}.$ We consider the bounding functions $F_i(\rho_i,\rho_i)$: 
 \begin{align*}
F_i(\rho_i,\rho_i)&=1-\rho_i-\left(1-\frac{\mu_{i,1}\rho_i}{n_1}\right)^{n_1}\left(1-\frac{\mu_{i,2}\rho_i}{n_2}\right)^{n_2}\\&= 1-\rho_i-\exp\left(-(\mu_{i,1}+\mu_{i,2})\rho_i-O\left(\frac{\mu_{i,1}^2\rho_i^2}{n_1}+\frac{\mu_{i,2}^2\rho_i^2}{n_2}\right)\right),
\end{align*}
by the Taylor-expansion of the natural logarithm around $1$. Since $\mu_{i,1}\leq 1+2\varepsilon$ and $\mu_{i,2}\leq 1+2\varepsilon,$ by the conditions of Theorem~\ref{mainresult3} and the fact that $\rho_i\leq 1$ (since it is a probability), we have \begin{align*}
F_i(\rho_i,\rho_i)&=1-\rho_i-\exp\left(-\left[\mu_{i,1}+\mu_{i,2} + O\left(n_2^{-1}\right)\right]\rho_i\right)\\&=1-\rho_i-\exp\left(-(1+\varepsilon_i)\rho_i\right),
\end{align*}
where $\varepsilon_i=\mu_{i,1}+\mu_{i,2}-1+O\left(n_2^{-1}\right)\sim\varepsilon$, by Conditions~\eqref{asymptoticalConditionConst}~and~\eqref{sumExpectationConst}. We set $D=\R_{>0}\times (0,1)$ and a real function $f$ on $D$ by setting 
$$f(x,\rho)=1-\rho-\exp(-x\rho),$$ for $(x,\rho)\in D.$ Note that we have $F_i(\rho_i,\rho_i)=f(\varepsilon_i,\rho_i).$

 It is well-known that $f(x,\rho)=0$ has exactly one solution for any fixed $x>0.$ Furthermore note that the partial derivative with respect to the variable $c$ of $f$ does not vanish on $D$, therefore we can apply the classical implicit function theorem in $\R^2.$ We consider $x=\varepsilon$ and denote by $(\varepsilon,\rho_\varepsilon)$ the corresponding solution of $f=0.$ Hence, there is an open set $U$ with $\varepsilon\in U$ and an open set $V$ with $\rho_\varepsilon\in V$ such that $$\left\{(u,g(u))\mid u \in U\right\}=\left\{(u,v)\in U\times V\mid f(u,v)=0\right\},$$ where $g$ is a continuous function on $U$ with $\rho_\varepsilon=g(\varepsilon)$. Let $i\in\{1,2\}.$ Because $\left|\varepsilon_i-\varepsilon\right|=o(1),$ we know that $\varepsilon_i\in U$ for large enough $n_1$ and $n_2,$ and this implies that $f(\varepsilon_i,g(\varepsilon_i))=0.$ Since $g$ is continuous we have 
$$
g(\varepsilon_1)\sim g(\varepsilon_2)\sim \rho_\varepsilon, 
$$
and it is sufficient to show that $\rho_1\le g(\varepsilon_1)$ and $\rho_2\ge g(\varepsilon_2)$.

For this last step, assume towards contradiction that $\rho_1>g(\varepsilon_1)$ and observe that 
$$
f(\varepsilon_1,\rho)<0, \forall\rho\in (g(\varepsilon_1),1].
$$
Since $(\rho_1,\rho_2)$ is by definition a solution of~\eqref{PGF} we have
$$0=F_1(\rho_1,\rho_2)\leq F_1(\rho_1,\rho_1)=f(\varepsilon_1,\rho_1)<0,$$ 
a contradiction. Analogously $\rho_2<g(\varepsilon_2)$ leads to a contradiction since 
$$
f(\varepsilon_2,\rho)>0, \forall\rho\in (0,g(\varepsilon_2)).
$$ Thus we have 
\begin{equation}\label{survivalProbConst}
\rho_1\sim\rho_2\sim \rho_\varepsilon.
\end{equation}

The remainder of the proof follows the lines of the proof of Theorem~\ref{mainresult} in Sections~\ref{dualSection},~\ref{widthSection}~and~\ref{proof}, by replacing $\rho_1\sim\rho_2\sim 2\varepsilon$ with statement~\eqref{survivalProbConst}.
\end{proof}

\section{Subcritical regime}\label{sec:subcritical}

In the subcritical regime, where the distance to the critical point is a constant, one can obtain a strong upper bound on the size of all components by a standard application of large deviation inequalities.

\begin{theorem}\label{mainresult4}
For $n_1\in\N$ and $n_2\in \N$ with $n_1\geq n_2$, let $n=n_1+n_2$ and let $1>\varepsilon>0$ be a fixed constant. Furthermore, let $$P=\left(p_{i,j}\right)_{i,j\in \{1,2\}}\in(0,1]^{2\times 2}$$ be a symmetric matrix of probabilities satisfying the following conditions:
\begin{equation}\label{sumExpectationSubConst}
\mu_{\iota,1}+\mu_{\iota,2}=1-\varepsilon+o(1), \text{ for any }\iota\in\{1,2\}, 
\end{equation}where $\mu_{i,j}=p_{i,j}n_j$ for every pair $(i,j)\in \{1,2\}^2$.
Then we have whp
\begin{equation*}
L_1\left(G(\mathbf{n},P)\right)=O(\log n).
\end{equation*}
\end{theorem}

\begin{proof}
Let the conditions be as in Theorem~\ref{mainresult4}. We fix a vertex $v$ and explore its component $\mathcal{C}_v$ in $G(\mathbf{n},P)$. Denote the resulting spanning tree by $\mathcal{T}_v$ and couple this process with a $2$-type branching process $\mathcal{T}_{\mathbf{n},P}$ as in Lemma~\ref{coupling}(i) such that $\mathcal{T}_v\subset\mathcal{T}_{\mathbf{n},P}.$ Let $\mathcal{S}_L$ be the event that $G(\mathbf{n},P)$ contains a component of size at least $L$ for some appropriately chosen real function $L.$ We want to show that $$\Prob\left(\mathcal{S}_L\right)=o(1).$$

Let us denote the (possibly infinite) sequence of vertices born in $\mathcal{T}_{\mathbf{n},P}\,$, with respect to the breadth-first exploration, by $\sigma=(v_1,v_2,v_3,\dots),$ where $v_1=v.$ 

For any vertex $u\in V_1\cup V_2$ let $X_u$ be the random variable that counts the number of children of $u$ and has a distribution $Bin(n_1,p_{j,1})+Bin(n_2,p_{j,2}),$ where $j\in\{1,2\}$ is the type of $u.$ Then consider the random variables 
\begin{equation*}
X_{v,L}:=\sum_{r=1}^{\min\{L,\left|\sigma\right|\}}X_{v_r}\leq\sum_{r=1}^{L}X_{v_r}=:X_{v,L}^* ,
\end{equation*} 
where $\left\{v_{\left|\sigma\right|+1},\dots,v_{L}\right\}$ is an arbitrary sequence of distinct additional vertices. Notice that $X_{v,L}^*$ is a sum of independent Bernoulli random variables whose expectation satisfies  
\begin{equation}\label{lastEstimate}
\left|\E\left(X_{v,L}^*\right)-L(1-\varepsilon)\right|\leq\gamma,
\end{equation} for some $\gamma=\gamma(n)=o(L),$ by Condition~\eqref{sumExpectationSubConst}. Hence, by application of a Chernoff bound (e.g.~\cite{JansonLuczakRucinskiBook}, page 29) we get
\begin{align}
\nonumber\Prob\left(X_{v,L}^*\geq L-1\right)&\stackrel{\eqref{lastEstimate}}{\leq}\Prob\left(X_{v,L}^*\geq\E\left(X_{v,L}^*\right) + \varepsilon L -1 -\gamma \right)
\\\nonumber&\stackrel{\eqref{lastEstimate}}{\leq} \exp\left(-\frac{(\varepsilon L-1-\gamma)^2}{2\left(L\left(1-\varepsilon\right)+\gamma+1/3\left(\varepsilon L-1-\gamma\right)\right)}\right)
\\&\stackrel{\gamma=o(L)}{\leq}\exp\left(-\frac{\varepsilon^2}{2-\frac{4\varepsilon}{3}}L(1+o(1))\right),\label{lastlastEstimate}
\end{align}  
uniformly for all vertices $v\in V_1\cup V_2$.

 In order to complete the proof we observe that the event $\left|\mathcal{T}_{\mathbf{n},P}\right|\geq L$ implies the event $X_{v,L}\geq L-1$ and therefore we get by application of the union bound
\begin{align*}
\Prob\left(\mathcal{S}_L\right)&\leq \sum_{v\in V_1\cup V_2}\Prob\left(\left|\mathcal{C}_v\right|\geq L\right)\leq \sum_{v\in V_1\cup V_2}\Prob\left(\left|\mathcal{T}_{\mathbf{n},P}\right|\geq L\right)\\
&\leq \sum_{v\in V_1\cup V_2}\Prob\left(X_{v,L}\geq L-1\right)
\leq \sum_{v\in V_1\cup V_2}\Prob\left(X_{v,L}^*\geq L-1\right)\\
&\stackrel{\eqref{lastlastEstimate}}{\leq}\exp\left(\log n-\frac{\varepsilon^2}{2-4\varepsilon/3}L(1+o(1))\right)=o(1),
\end{align*} 
for any $L>3\varepsilon^{-2}\log n,$ completing the proof.
\end{proof}

\section{Discussion}

In the previous sections we showed that the emergence of the giant component in the $2$-type random graph $G(\mathbf{n},P)$ is very similar to the behaviour of the binomial random graph $G(n,p)$, at least when each row of the expectation matrix is scaled similarly. In theory one therefore  could study $G_k(\mathbf{n},P),$ the $k$-type version of $G(\mathbf{n},P),$ assuming that each row of the expectation matrix sums up to approximately $1+\varepsilon$. It is to be expected that in this case we would have $\rho_1\sim\dots\sim\rho_k\sim 2\varepsilon$ and thus also a unique largest component of size $L_1\left(G_k(\mathbf{n},P)\right)\sim2\varepsilon n.$ Proving this for all $k\ge 3$ would be cumbersome at best, since for instance in our proof the bound on the total expected number of offspring of the dual process relies on explicitly calculating a set of generating functions.

Therefore let us take another perspective: Imposing the row-sum conditions ensures that the Perron-Frobenius eigenvalue of the offspring expectation matrix $M$ is roughly $1+\varepsilon$, however it also implies that the corresponding normalised left-eigenvector is not necessarily equal but close to $k^{-1}(1,\dots,1)$. In this spirit we could consider $G_k(\mathbf{n},P)$ for offspring expectation matrices $M$ whose Perron-Frobenius eigenvalue is $1+\varepsilon$ with the corresponding normalised positive left-eigenvalue $\mathbf{v}$ and study how the survival probabilities behave asymptotically. The Perron-Frobenius theory (c.f.~Chapter~V.6 in~\cite{AthreyaNeyBook}) provides a heuristic for this since the properly rescaled offspring vector of generation $t$ of the corresponding branching process converges almost surely to $\mathbf{v}$ as $t\to\infty,$ under the assumption that it survives.
Thus it would be interesting to know whether in this case it is true that $(\rho_1,\dots,\rho_k)\sim \beta_\varepsilon \mathbf{v}$ for some real function $\beta_\varepsilon=o(1)$. 

\section*{Acknowledgement}

We would like to thank Oliver Riordan for helpful discussions and valuable remarks. We are also very grateful to the referee for her/his useful suggestions.

\printbibliography

\end{document}